\documentclass[twoside]{article}
\usepackage{amsfonts}
\usepackage{stmaryrd}
\usepackage{amssymb}
\usepackage{euscript}
\usepackage{pstricks}
\usepackage{amsthm}
\usepackage{latexsym,amsmath,amscd}
\usepackage{mathrsfs}
\usepackage{graphicx}
\usepackage{color}
\usepackage{dsfont}
\usepackage{bm}
\usepackage{pgf,tikz}
\usetikzlibrary{arrows}

\newtheorem{theorem}{Theorem}[section]
\newtheorem{definition}{Definition}[section]
\newtheorem{proposition}{Proposition}[section]

\newtheorem{lemma}{Lemma}[section]
\newtheorem{corollary}{Corollary}[section]
\newtheorem{remark}{Remark}[section]

\numberwithin{equation}{section}

\let\originalleft\left
\let\originalright\right
\renewcommand{\left}{\mathopen{}\mathclose\bgroup\originalleft}
\renewcommand{\right}{\aftergroup\egroup\originalright}

\newcommand{\K}{\mathbb{K}}

\newcommand{\R}{\mathbb{R}}

\newcommand{\E}{\mathbb{E}}

\newcommand{\D}{D}
\renewcommand{\H}{\mathbb{H}}

\newcommand{\Oo}{\mathbb{O}}
\renewcommand{\P}{\mathbb{P}}

\newcommand{\FF}{\mathbb{F}}

\newcommand{\EP}{ {\mathbb E}_{\mathbb{P}}}

\newcommand{\Aset}{\mathcal{N}}

\newcommand{\Filt}{\mathcal{F}}

\newcommand{\Mset}{\mathcal{M}}
\newcommand{\Eset}{\mathcal{E}}
\newcommand{\Strat}{\mathcal{S}}

\renewcommand{\Game}{\mathcal{G}}

\newcommand{\Pmat}{\mathcal{P}}
\newcommand{\Zmat}{\mathcal{Z}}
\newcommand{\Kmat}{\mathcal{K}}
\newcommand{\Kmot}{\mathcal{K}^\dagger}
\newcommand{\Pmot}{\mathcal{P}^\dagger}

\newcommand{\esssup}{\operatornamewithlimits{ess\,sup}}
\newcommand{\essinf}{\operatornamewithlimits{ess\,inf}}

\let\inf\relax \DeclareMathOperator*\inf{\vphantom{p}inf}
\let\essinf\relax \DeclareMathOperator*\essinf{\vphantom{p}ess\,inf}

\newcommand{\map}[2]{\,{:}\,#1\rightarrow#2}
\newcommand{\norm}[1]{\left\|#1\right\|}

\newcommand{\proj}[2]{\pi_{#2}\left(#1\right)}

\newcommand\I{\mathds{1}}

\newcommand{\si}{\sigma}

\newcommand{\Gt}{\widetilde G}

\newcommand{\VI}{\operatorname{VI}}
\newcommand{\LCP}{\operatorname{LCP}}

\newcommand{\ZRG}{\operatorname{ZRG}}
\newcommand{\GRG}{\operatorname{GRG}}
\newcommand{\AG}{\operatorname{AG}}

\newcommand{\SAG}{\operatorname{AG}}
\newcommand{\ASG}{\operatorname{SAG}}
\newcommand{\SOL}{\operatorname{SOL}}

\newcommand{\wt}{\widetilde}
\newcommand{\wh}{\widehat}

\newcommand{\whts}{\widehat \tau^*_t}
\newcommand{\whD}{\widehat{D}}

\renewcommand{\bm}{}

\newcommand{\BSDE}{\operatorname{BSDE}}

\newcommand{\Keywords}[1]{\par\noindent{\small{\bf Keywords\/}: #1}}
\newcommand{\Class}[1]{\par\noindent{\small{\bf Mathematics Subjects Classification (2010)\/}: #1}}

\title{{\Large \bf STOCHASTIC MULTI-PLAYER COMPETITIVE \\ GAMES IN DISCRETE TIME}\vskip 90 pt}

\author{Ivan Guo and Marek Rutkowski\footnote{The research of Ivan Guo and Marek Rutkowski was supported under Australian Research Council's Discovery Projects funding scheme (DP120100895).}
\\ School of Mathematics and Statistics
\\ University of Sydney
\\ NSW 2006, Australia}

\date{\vskip 50 pt 6 April 2014 \vskip 15 pt}

\begin{document}
\maketitle
\vskip 45 pt
\begin{abstract}
A new class of  multi-player competitive stochastic games in discrete-time with an affine specification of
the redistribution of payoffs at exercise is proposed and examined. Our games cover as a very special
case the classic two-person stochastic stopping games introduced by Dynkin \cite{Dynkin}.
We first extend to the case of a single-period deterministic affine game the results from Guo and
Rutkowski \cite{Guo1,Guo2} where a particular subclass of competitive stopping games was studied.
We identify conditions under which optimal equilibria and value for a multi-player competitive game
with affine redistribution of payoffs exist. We also examine stochastic multi-period affine games and we show
that, under mild assumptions, they can be solved by the backward induction.
\vskip 20 pt
\Keywords{multi-player game, redistribution game, affine game, Dynkin stopping game, stochastic game, Nash equilibrium, optimal equilibrium}
\vskip 20 pt
\Class{91A06,$\,$91A15,$\,$60G40}
\end{abstract}


\newpage

%
%
%
%
%
%
%
%

\section{Introduction}

The classic \emph{Dynkin game}, introduced in the path-breaking paper \cite{Dynkin}, is a zero-sum, optimal stopping game between two players where each player can stop the game for a payoff observable at that time. They were the object of several papers published during the last forty years; see,  for instance, \cite{Cvitanic,Hamadene2,Laraki,Laraki2,Peskir2,Rosenberg,Touzi}. The classic zero-sum two-player Dynkin games were also used to construct and analyze the financial contracts dubbed the {\it game options} (also known as {\it Israeli options}). This notion was formally defined by Kifer \cite{Kifer1}, who proved the existence and uniqueness for the arbitrage price of a game option in some benchmark financial models.
For an exhaustive overview of results on two-player Dynkin games and Israeli options, the interested reader is referred to the recent survey by Kifer \cite{Kifer2}. Several alternative formulations of extended Dynkin games with more than two players can be found in the existing literature (see \cite{HH-2011,HH-2012,Karatzas1,Solan,Solan1}). For instance, Solan and Vieille \cite{Solan} introduced a multi-player quitting game, which terminates as soon as any player chooses to quit; then each player receives a payoff depending on the set of players quitting the game. Under certain payoff conditions, a subgame perfect uniform $\epsilon$-equilibrium using cyclic strategies can be found. Another version of a multi-player Dynkin was examined in Solan and Vieille \cite{Solan1} where players are given the opportunity to stop the game in a turn-based fashion. A subgame perfect $\epsilon$-equilibrium was once again shown to exist and consist of pure strategies when the game is non-degenerate.

More recently, Guo and Rutkowski \cite{Guo1,Guo2} introduced a class of $m$-player stopping games (dubbed the {\it redistribution games}) with a focus on designing the explicit dependencies between the payoffs of all players and their stopping decisions. The goal was to model a multilateral contract where all players are competing for a predetermined (albeit possibly random) total wealth. Each player can either exit (and thus also terminate) the contract for a predetermined benefit, or do nothing and receive an adjusted benefit, which reflects the discrepancies caused by any exiting decisions of other players. These adjustments were judiciously designed to ensure that the total wealth redistributed was fixed.

In the present work, we first generalize results from \cite{Guo1,Guo2} to the case of a single-period deterministic {\it affine game}. It should be stressed that in a multi-period framework, stochastic multi-player competitive games with an affine redistribution of payoffs are defined in a recursive way. Specifically, the payoffs redistribution at the moment when a game is first `stopped' or `exercised' at time $t$ by one of the players is based in the first place on the running payoffs for all players, but they also take into account the values for exercising players of an equivalent game, which is not exercised at time $t$, but instead it continues at least till time $t+1$. The idea of the recursive specification of the multi-player game is reminiscent of the concepts of \emph{exercise payoff} and \emph{continuation value} for some financial derivatives, such as American or game options (see \cite{Karatzas,Kifer1,Kifer2}), which indeed was the original motivation for our research. It is worth stressing that the interpretation of a player's decision to `exercise' depends on the context and, in practice, does not necessarily mean that a
 game is actually stopped.

 For instance, in our applications of results from this paper to multi-person financial contracts with a fixed maturity date, by `exercising' we mean a decision to `put' a traded tranche of a contract to the issuer. According to this interpretation, all tranches of a contract will be traded till the contract's maturity date, also in the case when some agents will decide to `exercise' the contract before its maturity. Hence we deal here with multi-player multi-exercise games, rather than typical stopping games in which exercise payoffs do not depend on an optimal behavior of players in the future. For applications of results from this work to arbitrage pricing of multi-person financial contracts, the interested reader is referred to~\cite{Guo3}.

 Let us only stress here that a financial derivative with meaningful decisions of counterparties should not be confused with a `game' in the usual sense of this term. Indeed, the standard rationale that underpins the game theory is that players should search for their respective optimal strategies, or at least for an equilibrium strategy profile for all players.
By contrast, the valuation problem for financial derivatives in an arbitrage-free market model hinges on the concept of replication (or super-hedging) of a contract's cash flows through the wealth process of a self-financing trading strategy. Therefore, an essential step in linking the \emph{arbitrage valuation} problem for financial derivatives to the \emph{equilibrium paradigm} prevailing in game theory relies on showing that the arbitrage prices, if well-defined, may be also interpreted as the values of a virtual game played by the counterparties under the martingale measure, which stems from the so-called \emph{Fundamental Theorem of Asset Pricing}. As recently shown in \cite{Guo3}, a necessary condition for a general multi-person game to be consistent with arbitrage pricing theory is that the game admits an \emph{optimal equilibrium} (as opposed to a weaker notion of a Nash equilibrium) under the unique martingale measure in a complete and arbitrage-free market model. For this reason, we aim to identify classes of multi-player competitive games for which an optimal equilibrium exists. It was proven in \cite{Guo3} (see part (ii) in Theorem 4.2 and Proposition 4.3 therein) that the games considered in the present work are suitable models of multi-person game options with unique arbitrage prices.

Furthermore, in financial applications, which indeed motivated our study of affine games, there is no much sense to deal with {\it mixed}  (i.e., randomised) strategies, since randomisation of strategies would undermine the concepts of an arbitrage opportunity (i.e., a risk-free profit) and super-hedging (i.e., a complete elimination of risk). Therefore, only {\it pure} strategies are considered in what follows. To be more specific, although a strategy can be random, it will always be adapted to a given reference filtration describing the information flow, so that no additional randomisation of a strategy is allowed. For more details on methods of mathematical finance, the interested reader is referred to, e.g., \cite{MMM,KS,MR} and for a detailed study of connections between multi-player stochastic games and multi-person game contingent claims, we refer to \cite{Guo3}.

In the present work, we first identify conditions under which a deterministic single-period multi-person affine game possesses an optimal equilibrium, so that the value for the game exists. As explained above, the existence of an optimal equilibrium is a crucial property from the viewpoint of arbitrage pricing theory. Next, we show that all single-period results can be immediately applied to a stochastic extension where both terminal and exercise payoffs are random, as long as expectations are incorporated into the definitions of a solution and an optimal equilibrium. Finally, we examine the stochastic multi-period affine games and we show that they can be solved through the backward induction under mild assumptions.
These recursive games can be readily applied to multi-person financial game options, where the properties of an optimal equilibrium become imperative in the pricing arguments (see, in particular, Theorem 4.2 in \cite{Guo3}).
Apart from multi-person game options, the multi-player competitive games examined in this work may find applications in other economic and financial contexts, for example, as a consumption model with bounded resources (see, e.g., Ramasubramanian \cite{Ramasubramanian1,Ramasubramanian2}). Continuous-time versions of affine games studied here are also of interest and they will be studied elsewhere; see, e.g., Nie and Rutkowski~\cite{Nie} for continuous-time redistribution games and
the associated backward stochastic differential equations (BSDEs).

This paper is organized as follows.

In Section \ref{sec2}, we first construct the single-period $m$-player \emph{affine game} associated with a non-singular matrix $G$. The role of $G$ is to specify the redistribution of payoffs among non-exercising players. We note that single-period affine games are also closely related to a well-known class of optimisation problems, known as \emph{linear complementarity problems} (LCPs). In fact, we show that finding Nash equilibria in an affine game is equivalent to finding solutions to an associated LCP. Using techniques from the theory of LCPs, we are able to identify sufficient conditions on the matrix $G$ that ensure the existence of Nash equilibria, optimal equilibria, individual values and coalition values.
The main result in Section \ref{sec2} is Theorem \ref{thmjj30}, which shows that if $G$ is a $\Pmat$-matrix (i.e., a matrix with all principal minors positive), then the affine game has at least one Nash equilibrium and all Nash equilibria attain the same payoff. Furthermore, if $G$ is a $\Kmat$-matrix (i.e., a $\Pmat$-matrix with non-positive off-diagonal terms), then the single-period affine game is \emph{weakly unilaterally competitive} (WUC) in the sense of Kats and Thisse \cite{Kats}, and the Nash equilibrium payoff is also the unique value of the game.

In Section \ref{sec3}, we examine the class of affine games with singular matrices.  Theorem \ref{thmjr10} shows that if $G$ is a $\Pmot$-matrix (i.e., a matrix with a non-negative determinant and positive proper principal minors), then  the affine game has a Nash equilibrium and a unique Nash equilibrium payoff. Moreover, when $G$ is a $\Kmot$-matrix
 (i.e., a $\Pmot$-matrix with  non-positive off-diagonal terms), then the affine game is WUC and has a unique value. We also examine coalition values and we show in Section \ref{ysec1e} that the additivity property holds for a certain subclass of affine games.

 In Section \ref{sec4}, we introduce stochastic multi-period $m$-player affine games and we show that they can be solved by the backward induction. Results for single-period affine games, regarding the existence of Nash equilibria, optimal equilibria, individual values and coalition values, are extended to the class of stochastic multi-period affine games (see Theorem \ref{thmmultiperiodvaluea}). We conclude by presenting in Section \ref{sec5} the reflected BSDE associated with the multi-period affine game. In the appendix, we briefly revisit some of our previous results for redistribution games obtained in \cite{Guo2}  and we show that they can be recovered from results of this paper. We stress once again that we work throughout with {\it pure} strategies only, as will be clear anyway from the definitions of a strategy profile and an optimal equilibrium.

\subsection{Motivation: Market Games and Game Options}

Let us first provide a tentative economic rationale for the concept of an \emph{affine game}.
We stress once again the we use the term `exercise' to describe any essential decision of a player
that affects other individuals. Recall that
in a multi-player \emph{redistribution game}, introduced in  Guo and Rutkowski \cite{Guo2}, the payoff discrepancies caused by the exercising players are treated as an aggregated total before being redistributed into the payoffs of the remaining players. If we would like control the redistribution on a more granular level, that is, to individually specify how the discrepancy of each exercising player is redistributed, we must define a larger class of games. The class of affine games introduced in the present work encompasses the class of redistribution games; their aim is to describe the competition between players in a more flexible way than
in \cite{Guo2}. Each affine game, denoted hereafter by $\AG(X,P,G)$, is associated with some $m\times m$ matrix $G$, which effectively replaces the redistribution quotients and the weights introduced in \cite{Guo1,Guo2}. The following basic example illustrates an economic underpinning for the concept of a single-period deterministic game with `affine redistribution of payoffs',
which, for the sake of brevity, is henceforth dubbed an \emph{affine game.}
For the sake of concreteness, we discuss here a stylized market game between $m$ competing firms that need to make decisions about their respective business strategies. Needless to say that a range of real-life situations
to which our model can be applied is, of course, much broader.

The $m$ firms are assumed to share the same market for a particular product or service and each of them would like to maximize the value for shareholders, formally represented by the fundamental value, denoted as $V_i$. Each firm $i$ faces the choice between two business strategies:
\hfill \break (A) a safe strategy leading to a fixed value of $V_i = X_i$ (e.g., by simply continuing its current contracts with existing customers) or \hfill \break (B) a risky strategy, which results in a yet unspecified variable value, denoted as $V_i$ (e.g., by offering innovative products in order to attract more customers), which will depend, in particular, on choices made simultaneously  by the other firms. 

If all firms select option (B) then their \emph{forecasted values} are given by some predetermined values,
denoted as $P_1, \dots ,P_m$. However, the scenario that \emph{all} firms decide to choose (B) will not be an optimal solution
when there exists at least one firm for which $P_i <X_i$, since in that case
it will be clearly sub-optimal for this firm to choose (B) over (A). Let us denote by $\Eset $ the set of firms that choose option (A). As already stated, for each firm $i\in\Eset$, the value is fixed to be $V_i = X_i$ and, by convention, it is referred to as
the \emph{exercise value}. By contrast, for each firm $j\notin\Eset$, its value $V_j$  should be adjusted from the forecasted value $P_j$, so we need to specify how the discrepancy $V_j - P_j$ should be computed. In an \emph{affine game}, this adjustment is assumed to depend linearly on the differences $X_i-P_i$ for all $i\in\Eset$.

To explain intuitively the competitive nature of the market game, we consider a particular firm $i \in \Eset $ and we show how that decision affects the market share for firm $j \notin \Eset $. Note that in this step we assume that only firm $i$ has chosen option (A), so that $\Eset = \{ i\}$. Out of the relative value $X_i-P_i>0$, which is gained by firm $i$ by not selecting (B), a fixed fraction comes from another firm $j$. For concreteness, suppose that $G_{ii}>0$ is the profit per customer of firm $i$ and $-G_{ji}>0$ is the profit per customer of firm $j$ multiplied by the fraction of customers abandoning firm $j$ and moving to firm $i$, as a result of the decision of firm $i$ to choose (A). Then the ratio of changes in values between firms $i$ and $j$ is given by the ratio $G_{ji}/G_{ii}<0$. The above considerations lead to the following expression for the value of each firm $j \notin \Eset $
\begin{gather}\label{eqlm16x}
V_j = P_j + G_{ji} \big(G_{ii}\big)^{-1} \big( X_i - P_i \big) < P_j.
\end{gather}
In this step, we have specified a matrix $G$ describing the market relationships between $m$ firms. It is not yet clear,
however, what will happen if more than one firm decides to adopt strategy (A)

We will now show how to deal with a general situation for any number of firms in the set $\Eset $. We now consider the vectors of values and we postulate that the vector $V-P$ lies in the column space of the sub-matrix $G_{\cdot\Eset}$, which is obtained from $G$ by taking columns with indices $i \in \Eset$. This assumption is a natural extension of expression \eqref{eqlm16x} and, in fact, it can be shown to uniquely specify the vector $V$ of fundamental values for all firms. Indeed, by using the condition that $V_i=X_i$ for all $i \in \Eset$ and solving for $V$, we arrive at the following vector equation (see Lemma \ref{xlem1})
\[
V = P + G_{\cdot \Eset} \big(G_{\Eset\Eset}\big)^{-1} \big( X_\Eset - P_{\Eset} \big).
\]
This expression is exactly the payoff function of a single-period affine game, as will be formally introduced in Definition \ref{defjj08} below. It is also consistent, as shown in the appendix, with the specification of $V$ for redistribution
 games studied in \cite{Guo1,Guo2} and in fact it allows for a more flexible (in some sense, more firm-specific) schemes for redistribution of relative profits/losses. Let us finally observe that the inequality $P_i>X_i$ does not necessarily imply that firm $i$ will choose (B) since, as a result of adverse decisions of other firms sharing the same market, the value $V_i$ may fall below $X_i$.

The next logical step is to extend a single-period affine game to a multi-period \emph{stochastic affine  game} or $\ASG(X,G)$
(see Section \ref{sec4}). Similar to the multi-period redistribution game examined in \cite{Guo2}, a recursive formulation of the game appears to be natural. The intuition behind this specification is perhaps best explained by expanding upon the single-period market game described above. Suppose that the firms are playing the market game at time 0, but the forecasted value $P$ is taken to be the expected value if all firms choose (B) at time 0 and then an analogous game is played at time 1. This in turn depends on the expected value from time 2 and so on. Even though the fundamental value should be maximized at time 0 only, the market game now also depends on future scenarios at times $1,2, \ldots ,T$ and thus also on potential future decisions of all players. As a result, the game will have the shape of a \emph{multi-period competitive game}, as defined in Section \ref{sec4}.

A particularly appealing motivation for our study of multi-person competitive games comes from the theory of arbitrage pricing of
financial derivatives. As we argue in  \cite{Guo3}, for a large class of multi-person financial contracts, the fair valuation
can be formally reduced to finding an optimal equilibrium of the corresponding multi-player game under a unique martingale
measure for the underlying market model. In particular,  Theorem 4.2 and Proposition 4.3 in \cite{Guo3} demonstrate
that the games considered in the present work correspond to multi-person game options with unique arbitrage prices given by
expected payoffs under an optimal equilibrium.  The interested reader is referred to \cite{Guo3} for a thorough examination of the issue of arbitrage pricing of multi-person contracts, not restricted to the special case of contracts with an affine structure.


\section{Competitive Games with Nonsingular Matrices} \label{sec2}

 We first examine the simplest case of a single-period deterministic game. Consider a game with $m$ \emph{players},  enumerated by the indices $1,2,\ldots,m$. The set of all players is denoted by $\Mset$. Let $X\in\R^m$ be the \emph{exercise payoff} and let
 $P\in\R^m$ be the \emph{terminal payoff}. We denote by $\Strat = \{0,1\}^m $ the class of strategy profiles where 0 (resp., 1) corresponds to \emph{exercise} (resp., \emph{continue}). We denote by $V(s)\in\R^m$ the vector
 (all vectors are understood to be column vectors) of payoffs obtained by players when a strategy profile $s$ is carried out.

To specify explicitly the vector of payoffs, we introduce a matrix $G\in\R^{m\times m}$, such that the $i$th column of $G$ is formally attached to player $i$.  Let $\Eset(s) :=\{i\in\Mset: s^i=0\}$ be the set of exercising players.  For any $s \in \Strat $, let $G_{\cdot \Eset (s)}$ be the sub-matrix obtained from $G$ by taking columns with indices from $\Eset (s)$.
We propose the following definition of a single-period $m$-player competitive game with affine
redistribution of payoffs, dubbed an {\it affine game}. It extends the concept of an $m$-player redistribution game
introduced and examined in \cite{Guo1,Guo2} (see Definition 2.1 in \cite{Guo1} or Definition 4 in \cite{Guo2}).
Let us stress that single-period affine games are merely building blocks for multi-period stochastic affine games studied in Section~\ref{sec3}.

\begin{definition} \label{defjj08} {\rm
Let $X, P\in\R^m$ be fixed vectors and $G\in\R^{m\times m}$ be a matrix with a positive diagonal and non-zero principal minors. An \emph{$m$-player affine game} $\AG (X,P,G)$ is a single-period deterministic game in which each player $i$ can either choose to exercise ($s^i=0$) or not exercise ($s^i=1$). For any strategy profile $s=(s^1,\ldots,s^m) \in \Strat $, the payoff vector
$V(s)=(V_1(s), \dots , V_m(s))$  is given by
\begin{gather}   \label{eqlm16}
V(s) = P+G_{\cdot \Eset(s)} \big(G_{\Eset(s)\Eset(s)}\big)^{-1}(X_{\Eset (s)} -P_{\Eset (s)})
\end{gather}
where $\Eset(s) = \{i\in\Mset: s^i=0\}$ is the set of exercising players.}
\end{definition}

  Let us make a comment on the difference between the class of redistribution games examined in \cite{Guo1,Guo2} and a more general class of game with affine redistribution of payoffs. For concreteness, let us assume that $m=5$
 and three players decide to exercise at time 0. Then, in a  redistribution game, the `losses' of the two non-exercising players only depend on the aggregated `gain' of exercising players, whereas in a more general case of an affine game they may also depend
 on a distribution of `gains' among three exercising players. Hence, using the concept of an affine game, we are able to address with more precision and flexibility the individual relationships between non-exercising and exercising players, and not only between cohorts of players.
%
%
%

To sum up, in an affine game the redistribution of losses among non-exercising players takes
into account the exact structure of gains made by exercising players, whereas in a redistribution game
only the total gain of exercising players matters. It is thus clear that the former set-up has the ability
to cover a wider spectrum of real-life applications. Other specifications of multi-person games are also
of interest. For instance, in \cite{Guo2} we examined the case of \emph{quitting games}
but, unfortunately, the crucial requirement of the existence of an optimal equilibrium is typically not satisfied
when a quitting game is played in a stochastic environment.

From the mathematical perspective, it is interesting to note that formula \eqref{eqlm16} can be derived by postulating that,
for any strategy profile $s \in \Strat $, the payoff deviation $V(s)-P$  lies in the column space of the matrix $G_{\cdot \Eset (s)}$. Formally, we postulate that the payoff function $V : \Strat \to \R^m$ satisfies
\begin{equation} \label{eqjj01}
\left\{
\begin{array}
[c]{l}
V_i(s)=X_i, \quad \forall\,i\in\Eset(s), \medskip\\
V(s)-P = G a(s), \quad a (s) \in\R^m , \medskip\\
a_i (s) =0, \quad \forall\,i\notin\Eset(s).
\end{array}
\right.
\end{equation}
By scaling $a_i$ appropriately, we may and do assume, without loss of generality, that $G$ has only positive diagonal terms.
The following lemma furnishes an explicit formula for the payoff function $V$ that satisfies conditions \eqref{eqjj01}
for every strategy profile $s \in \Strat $.

\begin{lemma} \label{xlem1}
Let $G\in\R^{m\times m}$ be a matrix with a positive diagonal and non-zero principal minors. Then  the payoff function $V$ can be written in terms of $X, P$ and $G$ as in \eqref{eqlm16}.
\end{lemma}

\begin{proof}
The last two conditions in \eqref{eqjj01}
 imply that the payoff vector satisfies
\[
V(s)=P+G_{\cdot \Eset(s)} a_{\Eset(s)} (s)
\]
where in $a_{\Eset(s)} (s)$ we take terms with indices in $\Eset(s)$. From the first condition in \eqref{eqjj01},
we obtain $V_i(s)=X_i$ for all $i\in \Eset(s)$ and thus
\[
a_{\Eset(s)}(s) =\big(G_{\Eset(s) \Eset(s)}\big)^{-1}(X_\Eset-P_\Eset)
\]
since the sub-matrix $G_{\Eset(s) \Eset(s)}$ is non-singular. We conclude that, for every strategy profile $s \in \Strat$, the payoff $V(s)$ can be explicitly written as in \eqref{eqlm16}.
\end{proof}

In Guo and Rutkowski \cite{Guo2}, we introduced the concept of the multi-player \emph{general redistribution game}.
It was shown there that the payoff of the general redistribution game can be expressed in terms of
a suitable projection. A similar result will be established here for the $m$-player affine game. To this end,
we will need the following elementary lemma.

\begin{lemma} \label{lemjq01}
Let $\pi$ be the projection mapping induced by the inner product $\langle\cdot,\cdot\rangle$ in $\R^m$.
Let $x\in\R^m$ be any vector and $\K\subseteq\R^m$ be any closed, convex set. Suppose that $y\in\K$.
Then $y=\pi_{\K}(x)$ if and only if $\langle y-x, z-y\rangle \geq 0$ for all $z\in\K$.
\end{lemma}

\begin{proposition} \label{propjj09}
The payoff function $V(s)$ satisfies $V(s)\in\H_{\Eset(s)}$ and
\begin{gather}\label{eqjq101}
\big(G^{-1}(V(s)-P)\big)^{\mathsf T}(y-V(s))=0, \quad \forall \,y\in\H_{\Eset(s)},
\end{gather}
where $\H_{\Eset(s)} := \big\{ {x} \in \R^m :\,  x_i = X_i,\ \forall\, i\in \Eset(s)\big\}$.
If $G$ is a positive definite symmetric matrix, then
\begin{gather}\label{eqjq102}
V(s)=\pi^{G^{-1}}_{\H_{\Eset(s)}}(P)
\end{gather}
where the projection $\pi^{G^{-1}}_{\H_{\Eset(s)}} : \R^m \to \H_{\Eset(s)}$ is taken under the inner product $\langle x, y\rangle^{G^{-1}} = x^{\mathsf T} G^{-1} y$.
\end{proposition}

\begin{proof}
In view of \eqref{eqjj01},  condition \eqref{eqjq101} can be rewritten as
\begin{gather}\label{eqjq05}
a(s)^{\mathsf T}(y-V(s))=0 , \quad \forall \, y \in \H_{\Eset(s)},
\end{gather}
where $a(s)$ satisfies $a_i(s)=0$ for all $i\notin\Eset(s)$. For all $i\in\Eset$, $y_i-V_i(s)=X_i-X_i=0$. So \eqref{eqjq05} clearly holds. If $G$ is a positive definite symmetric matrix, then \eqref{eqjq102} follows immediately from the property of projection given recalled in Lemma \ref{lemjq01}.
\end{proof}

For the reader's convenience, we recall the definition of Nash and optimal equilibria. Let $\Strat^k$ (resp. $\Strat^{-k}$)
be the strategy set for player $k$ (resp. all other players). With a slight abuse of notation, we find it convenient to represent an arbitrary strategy profile $s$ as $(s^k, s^{-k})$ where $s^k \in \Strat^k$ and $s^{-k} \in \Strat^{-k}$.

\begin{definition} \label{defaa02} {\rm
(i) A strategy profile $\si= ( \si^{1},\ldots,\si^{m}) \in\Strat$ is called a \emph{Nash equilibrium}, or simply an \emph{equilibrium}, if no single player can improve her  payoff by altering her  own strategy, that is,
\begin{gather}\label{eqaa011}
V_k(\si^{k}, \si^{-k}) = \sup_{s^k\in\Strat^k} V_k(s^k, \si^{-k}), \quad \forall \, k\in\Mset .
\end{gather}
\noindent (ii) A Nash equilibrium $\si= (\si^{1},\ldots,\si^{m} ) \in\Strat$ is called an \emph{optimal equilibrium} whenever
\begin{gather}\label{eqaa021}
V_k(\si^k, \si^{-k}) = \inf_{s^{-k}\in \Strat^{-k}} V_k(\si^k, s^{-k}), \quad \forall \, k\in\Mset .
\end{gather}
\noindent (iii) The \emph{value} $V^*$ of the game is defined by
\begin{gather}\label{eqaa024}
V^*_k = \sup_{s^k\in\Strat^k} \inf_{s^{-k}\in \Strat^{-k}} V_k(s^k, s^{-k})
= \inf_{s^{-k}\in \Strat^{-k}} \sup_{s^k\in\Strat^k} V_k(s^k, s^{-k}), \quad \forall \, k\in\Mset,
\end{gather}
assuming the equality in \eqref{eqaa024} holds.}
\end{definition}

\begin{remark} {\rm Equality \eqref{eqaa021} is equivalent to $V_k(\si^{k}, \si^{-k}) \geq V_k(s^k, \si^{-k})$ for all $s^k\in\Strat^k$.
When combined with condition \eqref{eqaa011} of a Nash equilibrium, an optimal equilibrium $\si$ satisfies
\[
V_k(\si^k, \si^{-k}) = \inf_{s^{-k}\in \Strat^{-k}} V_k(\si^k, s^{-k})
= \sup_{s^k\in\Strat^k} V_k(s^k, \si^{-k}), \quad \forall \, k\in\Mset ,
\]
or, equivalently, for every $k \in \Mset $,
\[
V_k(\si^k, s^{-k})\geq V_k(\si^k, \si^{-k}) \geq V_k(s^k, \si^{-k}),
\quad \forall \, s^k\in\Strat^k, \, \forall \, s^{-k}\in \Strat^{-k}.
\]
It is easy to show that the existence of optimal equilibria implies the existence of the value.
 Moreover, the payoff of any optimal equilibrium coincides with the value.}
\end{remark}

The next natural step is to investigate the {\it weakly unilaterally competitive} (WUC) property introduced by Kats and Thisse \cite{Kats} and the existence of Nash and/or optimal equilibria in $\AG(X,P,G)$. Specifically, we will address the following questions:
\begin{itemize}
\item Which choice of $G$ ensures that $\AG(X,P,G)$ has the  WUC property for all $X$ and $P$?
\item Which choice of $G$ guarantees a Nash equilibrium in $\AG(X,P,G)$ for all $X$ and $P$?
\end{itemize}
We will first attempt to gain some preliminary insight by analyzing the case where only one player exercises.
For the first question, recall from \cite{Kats} that the WUC property requires that, for every $k, l \in\Mset$,
\begin{align*}
V_k(s^k,s^{-k}) > V_k(\si^k,s^{-k})\ &\implies \ V_l(s^k,s^{-k}) \leq
 V_l(\si^k,s^{-k}), \\
V_k(s^k,s^{-k}) = V_k(\si^k,s^{-k})\ &\implies \ V_l(s^k,s^{-k}) = V_l(\si^k,s^{-k}),
\end{align*}
for all $s^k,\si^k\in\Strat^k$ and $s^{-k}\in\Strat^{-k}$. Let us consider the strategy profiles $s$ and $s'$ corresponding to $\Eset(s)=\emptyset$ and $\Eset(s')=\{k\}$. By applying the WUC property, we obtain, for all $l \neq k$,
\begin{align*}
V_k(s) > V_k(s')\ &\implies \ V_l(s) \leq V_l(s'),\\
V_k(s) < V_k(s')\ &\implies \ V_l(s) \geq V_l(s'),\\
V_k(s) = V_k(s')\ &\implies \ V_l(s) = V_l(s').
\end{align*}
It is also clear that $V(s)=P$ and $V(s')=P+G_{\cdot k} (G_{kk})^{-1} (X_k-P_k)$. Since $G_{kk}> 0$, it follows that $G_{lk} \leq 0$
for all $l \ne k$ and thus every off-diagonal term in $G$ must be non-positive. This is simply a necessary condition for the
WUC property, and it is by no means a sufficient one. However, it does motivate the choice of $\Zmat$-matrices
(see Definition \ref{defjj20}).

For the second question, we once again consider the strategy profiles $s$ and $s'$ corresponding to $\Eset(s)=\emptyset$ and $\Eset(s')=\{k\}$. If $s'$ is a Nash equilibrium, then
\[
X_k = V_k(s') \geq V_k(s) = P_k
\]
and thus $a_k=(G_{kk})^{-1} (X_k-P_k) \geq 0$. Furthermore, it is clear that $V_i(s)\geq X_i$ for all $i\in\Mset$.
When we add these constraints to \eqref{eqjj01}, it is clear that we are facing the problem that is reminiscent to what is known as the {\it linear complementarity problem}.

\subsection{Linear Complementarity Problems} \label{secLCP}

We will now briefly review results for the linear complementarity problem, which will be used in what follows. For more details, the reader is referred to the monographs by Cottle et al. \cite{Cottle} and Facchinei and Pang \cite{Facchinei}.

\begin{definition} {\rm
Given $q\in\R^m$ and $M\in\R^{m\times m}$, the \emph{linear complementarity problem} $\LCP(q,M)$ is to search for
a vector $z\in\R^m$ satisfying:
\begin{equation*} 
\left\{
\begin{array}
[c]{l}
z\geq 0, \medskip\\
q+Mz\geq 0, \medskip\\
z^{\mathsf T}(q+Mz) = 0,
\end{array}
\right.
\end{equation*}
where the inequalities are taken component-wise.}
\end{definition}

\begin{remark} {\rm
It is common to denote $w=q+Mz$ and equivalently state the problem as follows: find vectors $z, w \in \R^m$ satisfying:
\begin{equation} \label{ecejj10}
\left\{
\begin{array}
[c]{l}
w=q+Mz, \medskip\\
z\geq 0,\, w\geq 0,\medskip\\
z^{\mathsf T} w = 0 .
\end{array}
\right.
\end{equation}
In the remainder of this paper, unless explicitly specified otherwise, we refer to the pair $(z,w)$ whenever a solution of $\LCP(q,M)$ is mentioned.}
\end{remark}

We will need the following definition, which is taken from Fiedler and Pt\'ak \cite{Fiedler}.

\begin{definition}\label{defjj20} {\rm
Let $M$ be a real square matrix. \hfill \break
(i) If the principal minors of $M$ are all positive, then it is a \emph{$\Pmat$-matrix}. \hfill \break
(ii) If the off-diagonal terms of $M$ are all non-positive, then it is a \emph{$\Zmat$-matrix}. \hfill \break
(iii) If $M$ is both a $\Pmat$-matrix and a $\Zmat$-matrix, then it is a \emph{$\Kmat$-matrix}.}
\end{definition}

%

The next result shows that the existence and uniqueness result for the solution of the $\LCP(q,M)$ holds whenever $M$ is a $\Pmat$-matrix. For the proof of Proposition \ref{proplc03}, see Section 3.3 in Cottle et al.~\cite{Cottle}.

\begin{proposition} \label{proplc03}
The problem $\LCP(q,M)$ has a unique solution $z\in \R^m$ for all $q\in \R^m$ if and only if $M$ is a $\Pmat$-matrix.
\end{proposition}

\begin{remark}{\rm
 It is worth noting that the arguments used in the proof of Proposition \ref{proplc03} may be adapted to the linear complementarity problem in a general rectangular region. Specifically, given $l, u\in\R^m$ and $M \in \R^{m\times m}$ with $l_i<u_i$ for all $i=1,2, \dots ,m$, the problem
\begin{equation*}
\left\{
\begin{array}
[c]{l}
w = q + Mz,\quad  
l\leq z\leq u, \medskip\\
\I_{\{z_i>l_i\}} \I_{\{w_i>0\}} =\I_{\{z_i<u_i\}} \I_{\{w_i<0\}}=0, \quad i=1,2,\ldots,m,
\end{array}
\right.
\end{equation*}
has a solution $(z,w)\in\R^m\times \R^m$ for all $q\in\R^m$ if and only if $M$ is a $\Pmat$-matrix. Furthermore, the bounds $l_i$ and $u_i$ may be set to $-\infty$ and $\infty$, respectively.}
\end{remark}

\begin{remark}{\rm
If we set $w=V(s)-X,\, q=P-X,\, z=a(s)$ and $M=G$, then $\LCP(q,M)$ resembles the system of equations associated with the affine game $\AG(X,P,G)$. It is still not clear, however, whether a solution of  $\LCP(q,M)$ necessarily corresponds to a Nash equilibrium
of the game $\AG(X,P,G)$. This important issue will be addressed in Theorem \ref{thmjj30} below.}
\end{remark}


The following result summarizes the well-known properties of a solution of $\LCP (q,M)$.

\begin{proposition} \label{propjp20}
Let us fix $q\in\R^m$ and $M\in\R^{m\times m}$. For any $z\in\R^m$, let us set $w=q+Mz$. \hfill \break
(i) The following statements are equivalent: \hfill \break
(a)  $(z,w)$ is a solution to $\LCP (q,M)$, \hfill \break
(b)  $(z,w)$ satisfies $z=\proj{z-w}{\R^m_+}$, \hfill \break
(c)  $(z,w)$ satisfies $w^{\mathsf T}(y-z) \geq 0, \forall\,y\in\R^m_+$. \hfill \break
Here the projection $\pi : \R^m \to \R^m_+$ is under the Euclidean norm and $\R^m_+=\{x\in\R^m:x\geq 0\}$. \hfill \break
(ii) Let $M$ be a positive definite symmetric matrix. If $(z,w)$ is a solution to $\LCP (q,M)$, then
\begin{equation}\label{eqjp22}
z=\pi^{M}_{\R^m_+}(-M^{-1}q),\quad w=\pi^{M^{-1}}_{\R^m_+}(q),
\end{equation}
where the projections $\pi^{M}_{\R^m_+}: \R^m \to \R^m_+$ and $\pi^{M^{-1}}_{\R^m_+} : \R^m \to \R^m_+$ are taken under the inner products $\langle x, y\rangle^{M} = x^{\mathsf T} M y$ and $\langle x, y\rangle^{M^{-1}} = x^{\mathsf T} M^{-1} y$, respectively.
\end{proposition}

\def\skip1{
\begin{proof}  We first prove part (i). \hfill \break
[(a)$\iff$(b)]$\, $ For $i=1,\ldots,m$,
\[
z_i-(\proj{z-w}{\R^m_+})_i=z_i-\max(z_i-w_i,0)=\min(w_i,z_i).
\]
We thus see that
\begin{align*}
z=\proj{z-w}{\R^m_+} \
\iff\ \min(w_i,z_i)=0,\ i=1,\ldots,m
\iff\ z\geq 0,\ w\geq 0,\ z^{\mathsf T}w = 0
\end{align*}
as required. \hfill \break
[(b)$\iff$(c)]$\, $ This follows immediately from the property of projection in Lemma \ref{lemjq01}. \hfill \break
[(a)$\iff$(d)$\iff$(e)]$\, $  Noting the symmetry between $z$ and $w$, it suffices to repeat the previous arguments.

For part (ii), we observe that, by statements (c) and (e) in part (i), the pair $(z,w)$ satisfies $w=q+Mz$ and
\begin{align}\label{eqjp25}
\big(M(M^{-1}q+z)\big)^{\mathsf T}(y-z) \geq 0, \quad \big(M^{-1}(w-q)\big)^{\mathsf T}(y-w) \geq 0,\quad \forall\, y\in\R^m_+.
\end{align}
Since $M$ and $M^{-1}$ are positive-definite symmetric matrices, formula \eqref{eqjp22} follows immediately from
another application of Lemma \ref{lemjq01} to \eqref{eqjp25}.
\end{proof}

}

\subsection{Subgame and Value} \label{subgg}

As was already mentioned, in order for the game $\AG(X,P,G)$ to enjoy the WUC property, it is necessary for the off-diagonal terms
to have opposite signs to the diagonal terms in the respective columns. The following result analyzes in more detail the connections between $\Pmat$-matrices (or $\Kmat$-matrices) and the properties of the affine game $\AG(X,P,G)$.

\begin{theorem} \label{thmjj30}
Suppose $X,P\in\R^m$ are arbitrary vectors and $G\in\R^{m\times m}$ is a $\Pmat$-matrix. Then:  \hfill \break
(i) The affine game $\AG(X,P,G)$ has at least one Nash equilibrium and all Nash equilibria of $\AG(X,P,G)$ attain the same payoff $V^*$.   \hfill \break
(ii) The Nash equilibrium payoff $V^*$ satisfies
\begin{gather}\label{eqjq103}
\big(G^{-1}(V^*-P)\big)^{\mathsf T}(y-V^*)\geq 0, \quad \forall\,y\in\Oo (X),
\end{gather}
where $\Oo (X): = \{ x \in \R^m :\,  x\geq X \}$. If $G$ is a positive definite symmetric matrix, then
\begin{gather}\label{eqjq104}
V^*=\pi^{G^{-1}}_{\Oo (X)}(P)
\end{gather}
where the projection $\pi^{G^{-1}}_{\Oo (X)} : \R^m \to \Oo (X) $ is taken under the inner product $\langle x, y\rangle^{G^{-1}} := x^{\mathsf T} G^{-1} y$.  \hfill \break
(iii) If $G$ is a $\Kmat$-matrix, then the affine game $\AG(X,P,G)$ has the WUC property and the Nash equilibrium payoff $V^*$ is also the unique value of the game.
\end{theorem}

Before establishing Theorem \ref{thmjj30}, we prove some auxiliary lemmas, which deal with the properties of subgames.

\begin{lemma}\label{lemjj24}
 Let $G$ be a real $m\times m$ matrix with positive diagonal entries and non-zero principal minors. If player $m$ exercises in the game $\AG(X,P,G)$, then the subgame amongst the players in $\Mset'=\Mset\setminus\{m\}$ is given by $\AG(X_{-m},V_{\Mset'}(\sigma),\Gt)$ where $\sigma\in\Strat$ corresponds to $\Eset(\sigma)=\{m\}$ and $\Gt$ is the $(m-1)\times(m-1)$ matrix defined by
\begin{gather}  \label{eqjj30}
\Gt_{ij}=G_{ij}-\frac{G_{im} G_{mj}}{G_{mm}} , \quad 1\leq i,j\leq m-1.
\end{gather}
In particular, if $s^m=0$, then $V_{\Mset'}(s)=\widetilde V(s^{-m})$ and $a_{\Mset'}(s)=\widetilde a(s^{-m})$ where $\widetilde V$ and $\widetilde a$ are the counterparts to $V$ and $a$ in the game $\AG(X_{-m},V_{\Mset'}(\sigma),\Gt)$.
\end{lemma}

\begin{proof}
Recall that the payoff is given by $V(s)=P+Ga(s)$. If $m\in\Eset(s)$, then player $m$ exercises and $V_m(s)=X_m$, so that
the equality $X_m=P_m+G_{m \cdot} a(s)$ holds. After rearranging, we obtain
\[
a_m(s)=\frac{X_m-P_m -G_{m\Mset'} a_{\Mset'}(s)}{G_{mm}}
\]
where $\Mset'=\Mset\setminus\{m\}$. Consequently, we may represent the vector of payoffs for players from $\Mset'$ as follows
\begin{align}
V_{\Mset'}(s)&=P_{\Mset'}+G_{\Mset'm}a_m(s)+G_{\Mset'\Mset'} a_{\Mset'}(s)\nonumber\\
&=P_{\Mset'}+\frac{G_{\Mset'm}}{G_{mm}}(X_m-P_m) +\left(G_{\Mset'\Mset'}- \frac{G_{\Mset'm}G_{m\Mset'}}{G_{mm}}\right) a_{\Mset'}(s)\nonumber\\
&=V_{\Mset'}(\si)+\left(G_{\Mset'\Mset'}- \frac{G_{\Mset'm}G_{m\Mset'}}{G_{mm}}\right) a_{\Mset'}(s) \label{eqjj240}
\end{align}
where the strategy profile $\sigma\in\Strat$ corresponds to $\Eset(\sigma)=\{m\}$. Note that we have used here the fact that $V(\si)=P+ G_{\cdot m}(G_{mm})^{-1}(X_m-P_m)$ is the vector of payoffs if only player $m$ exercises.
All assertions can now be easily deduced from equation \eqref{eqjj240}.
\end{proof}

It was shown in Lemma \ref{lemjj24} that if player $m$ exercises, then the subgame between the remaining players is an affine game associated with the matrix $\Gt$. The next lemma demonstrates that the matrix $\Gt$ in fact retains the useful properties of $G$.

\begin{lemma}\label{lemjj25}
Let $G$ be a real $m\times m$ matrix such that $G_{mm}\neq 0$ and let the $(m-1)\times(m-1)$
matrix $\Gt$ be given by \eqref{eqjj30}. Then: \hfill \break
(i) If $G$ is a $\Pmat$-matrix, then $\Gt$ is a $\Pmat$-matrix. \hfill \break
(ii) If $G$ is a $\Zmat$-matrix and $G_{mm}> 0$, then $\Gt$ is a $\Zmat$-matrix. \hfill \break
(iii) If $G$ is a $\Kmat$-matrix, then $\Gt$ is a $\Kmat$-matrix.
\end{lemma}

\begin{proof}
(i) We construct the $m\times m$ matrix $A$ by setting
$\displaystyle
A = \begin{pmatrix} \Gt & G_{\Mset' m} \\ 0 & G_{mm} \end{pmatrix}
$
or, more explicitly,
\begin{align*}
A = \begin{pmatrix}
G_{11}-\frac{G_{1m} G_{m1}}{G_{mm}}&\cdots&G_{1(m-1)}-\frac{G_{1 m} G_{m(m-1)}}{G_{mm}}&G_{1m} \\
\vdots& \ddots &\vdots&\vdots \\
G_{(m-1)1}-\frac{G_{(m-1)m} G_{m1}}{G_{mm}}&\cdots&G_{(m-1)(m-1)}-\frac{G_{(m-1) m} G_{m(m-1)}}{G_{mm}}&G_{(m-1) m}\\
0&\cdots&0& G_{mm} \end{pmatrix}.
\end{align*}
It is clear that $A$ can be obtained from $G$ via the column operations $A_{\cdot j}=G_{\cdot j}-\frac{G_{\cdot m} G_{mj}}{G_{mm}}$ for $j=1,\ldots,m-1$. Therefore, for any $\Eset\subseteq\Mset$ with $m\in\Eset$,
\[
\det \big( G_{\Eset\Eset}\big) = \det\big(A_{\Eset\Eset}\big)=G_{mm} \det \big( \Gt_{\Eset\setminus\{m\},\Eset\setminus\{m\}} \big).
\]
Since $G_{mm}>0$, the principal minors of $\Gt$ must all be positive, as required.
For part (ii), it suffices to observe that, for $i\neq j$ and $1\leq i,j \leq m-1$, we have that $\Gt_{ij}=G_{ij}-\frac{G_{im} G_{mj}}{G_{mm}}\leq G_{ij}\leq 0$. Finally, part (iii) is an immediate consequence of (i) and (ii).
\end{proof}

\begin{proof} [Proof of Theorem \ref{thmjj30}]
We first prove part (i). For any $s\in\Strat$, we may write $V(s)=P+G a(s)$ with $(V_i(s)-X_i)a_i(s)=0$ for all $i\in\Mset$.
We will now show that $s$ is a Nash equilibrium whenever $V_i(s)\geq X_i$ and $a_i(s)\geq 0$ for all $i\in\Mset$.

For any $i\notin\Eset(s)$, if player $i$ chooses to exercise instead, she should not be able to improve her  payoff if $s$ was a Nash equilibrium, which in turn means that $V_i(s)\geq X_i$. For any $i\in\Eset(s)$, we consider the case where player $i$ decides to not exercise; let the corresponding strategy profile be $s'$ so that $\Eset(s')=\Eset(s)\setminus\{i\}$. Then
\[
V_{\Eset(s)}(s)-V_{\Eset(s)}(s')=G_{\Eset(s)\Eset(s)}\big(a_{\Eset(s)}(s)-a_{\Eset(s)}(s')\big).
\]
On the left-hand side, for $j\in\Eset(s),\, j\neq i$, we have $V_{j}(s)-V_{j}(s')=X_j-X_j=0$. Therefore, if we solve for the $i$th component of $a_{\Eset(s)}(s)-a_{\Eset(s)}(s')$ using Cramer's rule, the expression can be easily simplified to
\[
a_{i}(s)-a_{i}(s') = \frac{(V_{i}(s)-V_{i}(s')) \det(G_{\Eset(s')\Eset(s')})}{\det(G_{\Eset(s)\Eset(s)})}.
\]
Recalling that $a_{i}(s')=0$ (because $i\notin\Eset(s')$), we obtain
\[
V_{i}(s)-V_{i}(s')  = \frac{ \det(G_{\Eset(s)\Eset(s)})}{\det(G_{\Eset(s')\Eset(s')})}\, a_{i}(s).
\]
Since $G$ has positive principal minors, we conclude that $V_{i}(s)\geq V_{i}(s')$ if and only if $a_{i}(s) \geq 0$.

 We thus see that a strategy profile $s$ is a Nash equilibrium if and only if $(z,w)=(a(s),V(s)-X)$ is a solution to $\LCP(P-X, G)$. By Proposition \ref{proplc03}, there is a unique solution pair $(z^*,w^*)$. Therefore, all Nash equilibria must attain the unique payoff value $V^*=V(s^*)=w^*+X$ and one such Nash equilibrium $s^*$ is given by:
$s^*_i=0$ if and only if $w^*_i=0$.

Part (ii) is a direct consequence of Proposition \ref{propjp20}. Indeed, it suffices to translate the problem by $X$ to obtain the required result.

Let us now prove part (iii). We will prove the WUC property by induction on the number of players. For two players, the WUC property
is easy to check. Consider the case of $m>2$ players. We will compare the strategy profiles $s$ and $s'$, where $k\in\Eset(s')$ and $\Eset(s)=\Eset(s')\setminus\{k\}$.

If $\Eset(s)=\emptyset$, then for all $l\neq k$
\[
V_l(s')-V_l(s)=\frac{G_{lk}}{G_{kk}} (X_k-P_k) = \frac{G_{lk}}{G_{kk}} \big( V_k(s')-V_k(s)\big).
\]
It is clear that the WUC condition holds since $G_{kk}>0\geq G_{lk}$.

If $|\Eset(s)|\geq 1$, by rearranging the player indices, we can assume, without loss of generality, that $m\in\Eset, m\neq k$. Then, by Lemma \ref{lemjj24}, the game $\AG(X,P,G)$ can be reduced to the subgame $\AG(X_{-m},V_{\Mset'}(\sigma),\Gt)$ over the set of player $\Mset'=\Mset\setminus\{m\}$, where $\Gt$ is the $(m-1)\times(m-1)$ matrix given by  \eqref{eqjj30}.
Lemma \ref{lemjj25} shows that $\Gt$ is also a $\Kmat$-matrix. Therefore, by the induction assumption, the subgame $\AG(X_{-m},V_{\Mset'}(\sigma),\Gt)$ is WUC. In particular, $V_l(s')-V_l(s)$ can be written as a negative multiple of $V_k(s')-V_k(s)$.

 We conclude that the game $\AG(X,P,G)$ is WUC. It was shown by Kats and Thisse \cite{Kats} that in a WUC game, all Nash equilibria are also optimal equilibria. Since all $\Kmat$-matrices are also $\Pmat$-matrices, by part (i), the game $\AG(X,P,G)$ must have an optimal equilibrium and hence a unique value.
\end{proof}

\section{Competitive Games with Singular Matrices} \label{sec3}

In the appendix, we analyze the connections between the affine game $\AG(X,P,D)$ and the general redistribution game $\GRG(X,P,\alpha)$ with $\sum_{i=1}^m \alpha_i< 1$, which was introduced in \cite{Guo1,Guo2}. Note, however, that Definition \ref{defjj08} cannot be applied to the zero-sum redistribution game $\ZRG(X,P,\alpha )$ examined in \cite{Guo1,Guo2}, since the corresponding matrix $\whD $ given by formula  \eqref{ceqjp02}  is singular when $\sum_{i=1}^m \alpha_i=1$. This motivates us to extend Definition \ref{defjj08} of an affine game to cover also the case of a singular matrix $G$.

\begin{definition} \label{defjr01} {\rm
Fix $X,P\in\R^m$ and let $G\in\R^{m\times m}$ be a matrix with positive diagonal and non-zero `proper' principal minors (so $\det(G)=0$ is allowed). An {\it $m$-player affine game} $\AG(X,P,G)$ is a single-period deterministic game in which each player $i$ can choose either to exercise ($s^i=0$) or not exercise ($s^i=1$). For any strategy profile $s=(s^1,\ldots,s^m)$, the payoff vector
$V(s)=(V_1(s), \dots , V_m(s))$ is given by, for every $k \in \Mset $,}
\[
V_i(s)=
\begin{cases}
X_i ,& i\in\Eset(s),\\
P_i+G_{i \Eset(s)} \big(G_{\Eset(s)\Eset(s)}\big)^{-1}(X_{\Eset(s)}-P_{\Eset(s)}), \quad &i\notin\Eset(s).
\end{cases}
\]
\end{definition}

If a square matrix $G$ is non-singular then Definition \ref{defjr01} is consistent with Definition \ref{defjj08}.
Therefore, all the results established so far and regarding the game $\AG(X,P,G)$ with a non-singular matrix $G$ still apply.
In this subsection, we will only focus on the case of a singular matrix $G$.

\begin{definition}\label{defjr02} {\rm
A square matrix $M$ is called a \emph{$\Pmot$-matrix} if it has non-negative determinant and positive proper principal minors. Furthermore $M$ is said to be a \emph{$\Kmot$-matrix} if it is a {$\Pmot$-matrix} as well as a {$\Zmat$-matrix}.}
\end{definition}

\begin{remark}{\rm
It is worth noting that the class $\Pmot$ of matrices is not identical to the well-known classes of $\Pmat_0$ and $\Pmat_1$, which are defined as follows \hfill \break 
(i) a $\Pmat_0$-matrix is one with non-negative principal minors, \hfill \break
(ii) a $\Pmat_1$-matrix is a $\Pmat_0$-matrix where exactly one of the principal minors is zero. \hfill \break
Indeed, it is clear that we have the proper inclusions $\Pmot \varsubsetneq \Pmat_1\cup\Pmat \varsubsetneq \Pmat_0$.}
\end{remark}

The following auxiliary result is borrowed from the monograph by Cottle et al. \cite{Cottle} (for part (i), see Theorem 3.4.4; for part (ii), see Theorem 4.1.13).

\begin{proposition} \label{propjr06}
Assume that $M$ is a $\Pmot$-matrix. Then: \hfill \break
(i) For a fixed $q\in\R^m$, if $\LCP(q,M)$ has at least one solution $(z,w)$, then all solutions of $\LCP(q,M)$ are unique in $w$. \hfill \break
(ii) If for some $q\in\R^m$, $\LCP(q,M)$ does not have a solution, then there exists a fixed $v\in\R^m$ satisfying
\[
v>0,\quad v^{\mathsf T} M = 0,
\]
such that $\LCP(q,M)$ has a solution if and only if $v^{\mathsf T}q\geq 0$.
\end{proposition}

The main result of this subsection is a counterpart of Theorem \ref{thmjj30}.

\begin{theorem} \label{thmjr10}
(i) If $G$ is a $\Pmot$-matrix, then $\AG(X,P,G)$ has a Nash equilibrium and a unique Nash equilibrium payoff.\hfill \break
(ii) If $G$ is a $\Kmot$-matrix, then $\AG(X,P,G)$ is WUC and has a unique value.
\end{theorem}

\begin{proof}
(i) Let us first consider the case where $\LCP(P-X,G)$ has at least one solution. By Proposition \ref{propjr06}(i), there exists a solution pair $(z,w)$ where $w$ is unique. The same argument as in Theorem \ref{thmjj30} (used for $\Pmat$-matrices) can now be used to show that $w$ is the unique Nash equilibrium payoff.

Assume now that $\LCP(P-X,G)$ does not have a solution. Then, by Proposition \ref{propjr06}(ii), there exists $v>0$ such that $v^{\mathsf T}G=0$ and $v^{\mathsf T}(P-X)<0$. Consider the strategy profile $s$ corresponding to $\Eset(s)=\Mset$. We will show that $s$ is in fact a Nash equilibrium.

To this end, it suffices to show that $X_i \geq V_i(s')$ for all $i\in\Mset$ where $s'=(s^{-i},1)$. Keep in mind that for $j\neq i$, we have $V_j(s')=X_j$. By Definition \ref{defjr01}, since $\Eset(s')=\Mset\setminus\{i\}\neq\Mset$, we may write
$V(s')=P+Ga(s')$ for some $a(s')$. Now
\[
v^{\mathsf T}(P-X)=v^{\mathsf T}(V(s')-G a(s')-X)=v^{\mathsf T}(V(s')-X)=v_i(V_i(s')-X_i).
\]
Since $v_i>0$ and $v^{\mathsf T}(P-X)<0$, we must have $X_i > V_i(s')$, as required.

\noindent (ii) The same argument from Theorem \ref{thmjj30}(iii) for $\Kmat$-matrices can be applied to $\Kmot$-matrices.
\end{proof}

\subsection{Coalition Value in a Zero-Sum Game} \label{zsec1e}

So far, the definition of the value referred to each individual player. We will now consider
 the case where a subset of players $\Aset\subseteq\Mset$ is playing as a \emph{coalition}, using the \emph{collective payoff} $V_\Aset(s)=\sum_{i\in\Aset} V_i(s)$. A natural way to define the value of the game for the coalition $\Aset$ is to set
\begin{equation} \label{newva}
V^{*}_{\Aset} := \sup_{s^\Aset\in\Strat^\Aset}\inf_{s^{-\Aset}\in\Strat^{-\Aset}} V_\Aset(s^\Aset, s^{-\Aset})=\inf_{s^{-\Aset}\in\Strat^{-\Aset}} \sup_{s^\Aset\in\Strat^\Aset} V_\Aset(s^\Aset, s^{-\Aset}),
\end{equation}
assuming, of course, that the second equality above holds. In general, the value does not necessarily satisfy the {\it additivity property} $V^{*}_{\Aset}=\sum_{i\in\Aset} V^{*}_{i}$. In the following preliminary result, we consider an arbitrary $m$-person
\emph{zero-sum game}, that is, a game such that $V_\Mset(s)=0$ for all $s\in\Strat$.

\begin{proposition} \label{propab01a1}
Suppose the game is zero-sum and has an optimal equilibrium $\si\in\Strat$ with the value $V^*=V(\si)$.
Then for any subset $\Aset\subseteq\Mset$, the following equality holds
\[
V_\Aset(\si) :=\sum_{i\in\Aset} V^{*}_{i}=V^{*}_{\Aset}
.
\]
\end{proposition}

\begin{proof}
Since $\si$ is an optimal equilibrium, each player $i\in\Mset$ can guarantee the payoff $V_i(\si)$; in other words,
$V_i(\si)=\inf_{s^{-i}\in\Strat^{-i}} V_i(\si^i, s^{-i})$. Hence the players from $\Aset$ can ensure $V_\Aset(\si)$ by collectively choosing  $\si^\Aset$, since
\[
V_\Aset(\si) = \sum_{i\in\Aset} \inf_{s^{-i}\in\Strat^{-i}} V_i(\si^i, s^{-i}) \leq \sum_{i\in\Aset} V_i(\si^\Aset, s^{-\Aset}) =  V_\Aset(\si^\Aset, s^{-\Aset}), \quad \forall\,s^{-\Aset}\in\Strat^{-\Aset}.
\]
Consequently,
\begin{gather}\label{eqab01a2}
V_\Aset(\si) \leq\inf_{s^{-\Aset}\in\Strat^{-\Aset}} V_\Aset(\si^\Aset, s^{-\Aset}) \leq \sup_{s^\Aset\in\Strat^\Aset}\inf_{s^{-\Aset}\in\Strat^{-\Aset}} V_\Aset(s^\Aset, s^{-\Aset}).
\end{gather}
If we apply the same argument to the player set $-\Aset$, we see that they can also guarantee their payoff $V_{-\Aset}(\si)$,
\[
V_{-\Aset}(\si) \leq V_{-\Aset}(s^\Aset, \si^{-\Aset}), \quad \forall\,s^{\Aset}\in\Strat^{\Aset}.
\]

Using the zero-sum condition, we obtain
\begin{gather}\label{eqab01a4}
V_\Aset(\si) = -V_{-\Aset}(\si)  \geq -V_{-\Aset}(s^\Aset, \si^{-\Aset}) = V_{\Aset}(s^\Aset, \si^{-\Aset}), \quad \forall\,s^{\Aset}\in\Strat^{\Aset},
\end{gather}
and thus
\begin{gather}\label{eqab01a3}
V_\Aset(\si) \geq \sup_{s^\Aset\in\Strat^\Aset} V_{\Aset}(s^\Aset, \si^{-\Aset}) \geq \inf_{s^{-\Aset}\in\Strat^{-\Aset}} \sup_{s^\Aset\in\Strat^\Aset} V_\Aset(s^\Aset, s^{-\Aset}).
\end{gather}
By combining \eqref{eqab01a2} and \eqref{eqab01a3} with the well-known inequality
\[
\inf_{s^{-\Aset}\in\Strat^{-\Aset}} \sup_{s^\Aset\in\Strat^\Aset} V_\Aset(s^\Aset, s^{-\Aset}) \geq \sup_{s^\Aset\in\Strat^\Aset}\inf_{s^{-\Aset}\in\Strat^{-\Aset}} V_\Aset(s^\Aset, s^{-\Aset}),
\]
we obtain the desired equality $V_\Aset(\si) = V^{*}_{\Aset} $.
\end{proof}

\subsection{Affine Games with Additive Values} \label{ysec1e}

In Theorems \ref{thmjj30} and \ref{thmjr10}, we established the existence of the value for
for each individual player in the game $\AG(X,P,G)$ where $G$ is a $\Kmot$-matrix. As in the preceding
subsection, we suppose that a given subset of players $\Aset\subseteq\Mset$ is playing as a coalition having in view the aggregated payoff function $V_\Aset(s)=\sum_{i\in\Aset} V_i(s)$. Proposition \ref{propab01a1} showed that in a zero-sum game with optimal equilibria, the value $V^{*}_{\Aset}$ (in the sense of formula \eqref{newva}) exists and satisfies the additive property $V^{*}_{\Aset}=\sum_{i\in\Aset} V^{*}_i(s)$. Although affine games are not necessarily zero-sum, we may introduce a dummy player to create a zero-sum extended game. Then we may apply Proposition \ref{propab01a1}, in order to show that the additivity property also holds for certain subclass of affine games.

\begin{theorem} \label{thmjt01}
Consider the affine game $\AG(X,P,G)$ where $X, P\in\R^m$ and $G$ is a $\Kmot$-matrix. Suppose that the column sums of $G$ are non-negative, that is,
\begin{gather} \label{eqjt01}
\sum_{i\in \Mset } G_{ij} \geq 0,\quad \forall\,j\in\Mset.
\end{gather}
Then for any $\Aset\subseteq\Mset$, the value $V^{*}_{\Aset}$ exists and is additive, in the sense that $V^{*}_{\Aset}
= \sum_{i\in\Aset} V^{*}_{i}$.
\end{theorem}

\begin{proof}
In order to create a zero-sum extended game, we introduce a dummy player with index $0$, which has the terminal payoff $P_{0}=-\sum_{i\in\Mset} P_{i}$ and is not allowed to exercise. By convention, we define her  exercise payoff to be $X_{0}=-\sum_{i\in\Mset} X_{i}$ and we denote $$\wt P=(P_0,P_1,\ldots,P_m),\ \wt X=(X_0,X_1,\ldots,X_m).$$
By appending an extra row and column to $G$, we form the matrix $\wt G$
\[
\wt G =
\begin{pmatrix}
0 & -\sum_{i=1}^m G_{i1} & \cdots & -\sum_{i=1}^m G_{im} \\
0 & G_{11}& \cdots & G_{1m} \\
\vdots & \vdots & \ddots & \vdots \\
0 & G_{m1} & \cdots &  G_{mm}
\end{pmatrix}.
\]
Note that since the dummy player has no right to exercise, the first column of $\wt G$ is in fact irrelevant;
 the column of 0s was chosen to simplify later arguments. From condition \eqref{eqjt01}, we see that
 the off-diagonal terms in $\wt G$ are all non-positive and thus $\wt G$ is a $\Zmat$-matrix.
We consider the extended affine game $\AG(\wt X, \wt P, \wt G)$ with the payoff function given by
(of course, $\Eset (s)$ does not contain the dummy player)
\[
V_i(s)=
\begin{cases}
X_i ,& i\in\Eset(s),\\
\wt P_i+ \wt G_{i \Eset(s)} \big( \wt G_{\Eset(s)\Eset(s)}\big)^{-1}(\wt X_{\Eset(s)}-\wt P_{\Eset(s)}), \quad &i\notin\Eset(s).
\end{cases}
\]

To see that $\AG(\wt X, \wt P, \wt G)$ is a zero-sum game, it suffices to note that the vectors $\wt P, \wt X$ as well as the columns of $\wt G$ lie in the subspace $\big\{ x\in\R^{m+1} : \sum_{i=0}^{m} x_i =0 \big\}$. It is also easy to see that for players $1,\ldots, m$, the payoff function of the extended game $\AG(\wt X, \wt P, \wt G)$ is identical to the original game $\AG( X,  P,  G)$. Therefore, we may set $\wt V_{0}(s)=-\sum_{i\in\Mset} V_i (s)$ for all $s\in\Strat$.

Let $s^*$ be an optimal equilibrium of the original game $\AG( X,  P,  G)$. Since the dummy player only has the action of `not exercising', it is clear that $s^*$ is a Nash equilibrium of $\AG(\wt X, \wt P, \wt G)$. In order to show that $s^*$ is also an optimal equilibrium of $\AG(\wt X, \wt P, \wt G)$, by the result of Kats and Thisse \cite{Kats},
it suffices to show that $\AG(\wt X, \wt P, \wt G)$ is WUC. Since $G$ is a $\Kmot$-matrix and the additional entries in $\wt G$ are non-positive, we may employ the same argument from Theorem \ref{thmjj30}(iii) to conclude that $\AG(\wt X, \wt P, \wt G)$ is indeed WUC. Finally, since $s^*$ is an optimal equilibrium in the zero-sum game $\AG(\wt X, \wt P, \wt G)$, the required result follows immediately from Proposition \ref{propab01a1}.
\end{proof}

\begin{remark} {\rm
Lemma \ref{appgrg} in the appendix shows that the redistribution game $\GRG (X,P, \alpha )$ (where $X,P, \alpha \in\R^m$, $a>0$ and $\sum_{i=1}^m \alpha_i < 1$) is equivalent to $AG(X,P,\whD)$ where $\whD$ is given by \eqref{ceqjp02}.
The sum of column $i$ in $D$ is given by
$$
\alpha_i - \sum_{j=1}^m \alpha_i\alpha_j = \alpha_i \Big(1-\sum_{j=1}^m \alpha_j \Big) \geq 0.
$$
Consequently, by Theorem \ref{thmjt01}, the general redistribution game $\GRG (X,P,\alpha)$ satisfies the additivity property $V^{*}_{\Aset}=\sum_{i\in\Aset} V^{*}_{i}(s)$ for all $\Aset\subseteq\Mset$.}
\end{remark}

\section{Multi-Period Stochastic Affine Games} \label{sec4}

We will now extend the notion of a multi-period $m$-player redistribution game introduced in \cite{Guo2}
(see Definition 5.4 therein). We assume that we are given an underlying probability space $(\Omega, \Filt, \P )$, which is endowed with the filtration $\FF = \{\Filt_t: t=0,1,\ldots,T\}$ representing the information flow. For brevity, we hereafter write $[t,T]= \{ t ,t+1 , \dots , T\}$.

\begin{definition}\label{defmultiperiodgame} {\rm
For each $t=0,1,\ldots,T$, a stochastic multi-period competitive game with an affine redistribution of payoffs,
dubbed a \emph{stochastic affine game} and denoted as $\ASG_t(X,G)$, is defined on the time interval $[t,T]$.
It is specified by the following inputs: \hfill \break
(a) the set of $m$ players $\Mset=\{1,2,\ldots,m\}$, \hfill \break
(b) the family of $\FF$-adapted processes $\bm{X}_t=(X^1_{t}, \ldots, X^m_{t})$ for $t=0,1, \dots ,T$, \hfill \break
(c) the $m\times m$ deterministic $\Kmot$-matrix $G$, \hfill \break
and the following rules of the game: \hfill \break
(i) each player can exercise at any time in the interval $[t,T]$ and the game stops as soon as anyone exercises; if no one exercises before $T$, then everyone must exercise at $T$, \hfill \break
(ii) the strategy $s^k_{t}$ of player $k$ is a random time chosen from the space $\Strat^k_{t}$ of $\FF$-stopping times
  with values in $[t,T]$; hence a strategy profile $s_t=(s^1_{t},\ldots,s^m_{t})\in\Strat_t$ is the $m$-tuple of $\FF$-stopping times, \hfill \break
 (iii) for each strategy profile $s_t \in \Strat_t $, the outcome of the game is the expected payoff vector $\bm{V}_t(s_t)=(V^1_{t}(s_t),\ldots,V^m_{t}(s_t))$, defined by
\begin{gather}\label{eqju11}
V^k_{t}(s_t)=\EP\big( X^k_{\widehat s_t} \I_{\{k\in\Eset(s_t)\}}+\wh X^k_{\widehat s_t}\I_{\{k\notin\Eset(s_t)\}} \,\big|\, \Filt_t \big)
\end{gather}
where $\widehat s_t = s^1_{t}\wedge \cdots \wedge s^m_{t}$ is the minimal stopping time and
$\Eset(s_t)=\{i\in\Mset: s^i_{t} = \widehat s_t\}$ is the random set of earliest stopping players; furthermore,
\begin{gather}\label{eqju12}
\wh X^k_{\widehat s_t}=V^{*k}_{\widehat s_t+1}
+G_{k \Eset(s_t)} \big(G_{\Eset(s_t)\Eset(s_t)}\big)^{-1}
\big( X^{\Eset(s_t)}_{\widehat s_t} - V^{*\Eset(s_t)}_{\widehat s_t+1}\big),\quad \widehat s_t<T ,
\end{gather}
where $\bm{V}_u^*=(V^{*1}_u,\ldots,V^{*m}_u)$ is the value of the game $\ASG_u(X,G)$ for $u=t+1,t+2, \dots ,T$.}
\end{definition}

As the game is stopped at time $\widehat s_t$, the indicator functions appearing in \eqref{eqju11} are aimed to separate the exercising players from the others: $X^{k}_{\widehat s_t}$ is the payoff for an exercising player, whereas $\wh X^{k}_{\widehat s_t}$ is the payoff for a non-exercising player. Intuitively, the game $\ASG_{\widehat s_t+1}(X,G)$ can be considered as the continuation of the current game if it is not stopped at time $\widehat s_t$. Note also that, in formula \eqref{eqju12}, the term $\wh X^{k}_{\widehat s_t}$ is not defined for $\widehat s_t=T$. This does not matter, however, because if the game is stopped at $T$, then every player must exercise and receive $X^{k}_{T}$, so that $\wh X^{k}_{T}$ is irrelevant.

\begin{remark} {\rm
It is obvious that Definition \ref{defmultiperiodgame} is in fact recursive. Since $\widehat s_t+1 > t$, the payoff of $\ASG_t(X,G)$ depends on the values of $\ASG_{t+1}(X,G), \ldots, \ASG_T(X,G)$, which can also be seen as subgames of $\ASG_t(X,G)$. It is also possible to view the multi-period game as a recursive sequence of embedded single-period games. If the game $\ASG_t(X,G)$ is stopped at $t$, then the exercising players receive $X^k_{t}$, while the other players receive
\begin{gather}\label{eqju13}
\E_\P \big(V^{*k}_{t+1} \,|\, \Filt_t \big) +G_{k \Eset} \big(G_{\Eset\Eset}\big)^{-1} \big( X^\Eset_{t} - \E_\P \big( V^{*\Eset}_{t+1} \,|\, \Filt_t\big) \big).
\end{gather}
If $\ASG_{t}(X,G)$ is not stopped at time $t$, then it reduces to the game $\ASG_{t+1}(X,G)$. And, finally, the game $\ASG_{T}(X,G)$ always stops at time $T$ as everyone exercises.}
\end{remark}

\begin{remark}{\rm
The stopping game described by Definition \ref{defmultiperiodgame} is perhaps not the most obvious generalization of the single-period  stochastic affine game. A more natural generalization would be for the non-exercising player $k$ to receive, for any $s_t \in \Strat_t$,
\begin{gather} \label{gamett}
X_{k,T} +G_{k \Eset(s_t)} \big(G_{\Eset(s_t)\Eset(s_t)}\big)^{-1}\big( X^{\Eset(s_t)}_{\widehat s_t} - X^{\Eset(s_t)}_{T}\big),\quad \widehat s_t<T,
\end{gather}
when the game is stopped, that is, using $X^{\Eset(s_t)}_{T}$ instead of the value $V^{*\Eset(s_t)}_{\widehat s_t+1}$.
However, even in deterministic cases, this does not always produce optimal equilibria in pure strategies.
For example, let us consider a game with
\[
\bm{X}_0=
\begin{pmatrix}
-1 \\
-1 \\
0
\end{pmatrix},\quad
\bm{X}_1=
\begin{pmatrix}
-2 \\
-2 \\
4
\end{pmatrix},\quad
\bm{X}_2=
\begin{pmatrix}
0\\
0\\
0
\end{pmatrix},\quad
G=
\begin{pmatrix}
2/9 & -1/9 & -1/9 \\
-1/9 & 2/9 & -1/9 \\
-1/9 & -1/9 & 2/9
\end{pmatrix}.
\]
It is clear that player 3 will always want to exercise at time 1, while there is a `prisoner's dilemma' between players 1 and 2 at time 0. Then the game given by \eqref{gamett} with $T=1$ has two Nash equilibria (with different payoffs), but no optimal equilibria in pure strategies.}
\end{remark}

\subsection{Optimal Equilibrium}  \label{sec4b}

For convenience, we introduce the following notation.

\begin{definition} {\rm
Suppose $G$ is a $\Pmot$-matrix and $X,P \in \R^m$ are given vectors. We denote by $\SOL(X,P,G)$ the unique Nash equilibrium payoff of the single-period affine game $\AG(X,P,G)$.}
\end{definition}

The notation of $\SOL$ was originally used to denote the solution of a linear complementarity problem. This is consistent with our usage, since $\SOL(X,P,G)$ is indeed the solution of $\LCP(P-X,G)$ if $G$ is non-singular. If $G$ is singular, then $\LCP(P-X,G)$ may not have a solution. Nevertheless we use $\SOL(X,P,G)$ to denote the Nash equilibrium payoff of $\AG(X,P,G)$, for convenience.
So far, we have used $V^*$ when referring to $\SOL(X,P,G)$. But in the upcoming discussions, $\SOL(X,P,G)$ is more appropriate, since it explicitly shows the dependence on $X$ and $P$. In the special case of $G$ being a positive-definite symmetric matrix, we have the equality $\SOL(X,P,G)=\pi_{\Oo (X)}^{G^{-1}}(P)$, as demonstrated in Theorem \ref{thmjj30}.

At this moment, it is not clear that the game introduced in Definition \ref{defmultiperiodgame} is well-defined, since it is not yet known whether the game $\ASG_t(X,G)$ always has a value. The following theorem shows that this is indeed the case.

\begin{theorem} \label{thmmultiperiodvaluea}
Let $X_t=(X^1_{t}, \ldots, X^m_{t})$ be $\FF$-adapted processes and $G$ be a deterministic $\Kmot$-matrix. Recursively define the $\Filt_t$-measurable vector $\bm{U}_t=(U^1_{t},\ldots,U^m_{t})$ by setting $\bm{U}_T:=\bm{X}_T$ and
\begin{gather}\label{eqju22}
\bm{U}_t := \SOL \big(X_t,\EP \big(\bm{U}_{t+1}\,|\,\Filt_{t}\big), G\big),\quad  t = 0,1, \dots , T.
\end{gather}
Define the $\FF$-stopping times $\tau^*_t =(\tau^{*1}_t,\ldots,\tau^{*m}_t)$ by
\begin{gather}\label{eqju23}
\tau^{*i}_t := \inf \big\{ u\in [t,T]: U^i_{u}=X^i_{u} \big\}.
\end{gather}
Then the following statements are valid: \hfill \break
(i)  the equality $\bm{U}_t=\bm{V}_t(\tau^*_t)$ holds for all $t$, \hfill \break
(ii) the strategy profile $\tau^*_t$ is an optimal equilibrium of $\ASG_t(X,G)$ and $U_t=V^*_t$ is the value.
\end{theorem}

 Before proceeding to the proof of Theorem \ref{thmmultiperiodvaluea}, we will state an auxiliary lemma in which
 we consider a single-period setting.
 We assume that we are given the $\Filt_1$-measurable random vectors $\bm{P}=(P_1, \ldots, P_m)$ and $\bm{X}=(X_1, \ldots, X_m)$ and a deterministic $m\times m$ matrix $G$ with non-zero proper principal minors,
that is, with $\det(G_{\Eset\Eset})\neq 0$ for all proper subsets $\Eset \subset \Mset$.
In an $m$-player single-period stochastic affine game $\SAG (X,P,G)$ each player can only exercise at time 0 and the payoffs are distributed at time 1. The space of strategy profiles is $\Strat = \prod_{i\in\Mset} \Strat^i$ where $\Strat^i=\{0,1\}$ is the space of strategies for player $i$. 
For any $s \in \Strat $, the outcome of the game is the expected payoff vector $\bm{V}(s)=(V_1 (s),\ldots,V_m (s))$, defined by	
\[
V_i(s)=\EP\big( X_i \I_{\{i\in\Eset(s)\}}+\wh X_i\I_{\{i\notin\Eset(s)\}} \big)
\]
where $\Eset(s)$ is the set of exercising players and
\[
\wh X_i=P_i+G_{i \Eset(s)} \big(G_{\Eset(s)\Eset(s)}\big)^{-1}\big(X_{\Eset(s)}-P_{\Eset(s)}\big).
\]
By the linearity of the payoff function $V_k(s)$, the stochastic game consider here is essentially equivalent the deterministic affine game $\AG( \EP(\bm{X}), \EP(\bm{P}), G)$. In view of Theorem \ref{thmjj30}, the proof of the next lemma is straightforward and thus it is omitted.

\begin{lemma} \label{thmexpectedsoln}
Consider the single-period stochastic affine game $\SAG (X,P,G)$. Then: \hfill \break
(i) If $G$ is a $\Pmot$-matrix, then $V^*=\SOL \big( \EP(\bm{X}), \EP(\bm{P})\big)$ is the unique Nash equilibria payoff and a possible Nash equilibrium $s^*=(s^{*}_{1},\ldots,s^{*}_{m})$ is given by
\[
s^{*}_{i}=0\ \iff \ \SOL \big( \EP(\bm{X}), \EP(\bm{P})\big)_i=\E_\P\left(X_i\right).
\]
(ii) If $G$ is a positive definite symmetric matrix, then $V^*=\pi^{G^{-1}}_{\Oo (X)}(\EP \left(\bm{P}\right))$
where $\Oo (X)$ is the orthant given by
\[
\Oo (X) := \big\{\bm{x}\in\R^m:\, x_i\geq \E_\P\left(X_i\right), \ 1\leq i \leq m \big\}
\]
and the projection $\pi^{G^{-1}}_{\Oo (X)}$ is taken under the inner product $\langle x, y\rangle^{G^{-1}} = x^{\mathsf T} G^{-1} y$. \hfill \break (iii) If $G$ is a $\Kmot$-matrix, then the game $\SAG(X,P,G)$ is WUC. All Nash equilibria are also optimal equilibria and $V^*$ is also the unique value.
\end{lemma}

\noindent{\it Proof of Theorem \ref{thmmultiperiodvaluea}.}
Throughout the proof, we will use the following properties: \hfill \break
(a) If the game $\ASG_t(X,G)$ is stopped at time $t$, then the payoff function coincides with the single-period  affine game $\AG(X_t,\EP({V}_{t+1}^*\,|\,\Filt_t), G)$. This follows from Definition \ref{defmultiperiodgame} and formula \eqref{eqju13}. \hfill \break
 (b) If the game $\ASG_t(X,G)$ is not stopped at time $t$ under some strategy profile $s_t \in\Strat_t$, then $s_t$ is also a strategy profile of $\ASG_{t+1}(X,G)$, so that $s_t \in\Strat^{t+1}$. Furthermore, by the definition of $V_t(s_t)$ in \eqref{eqju11}, the equality $V_t(s_t)=\EP (V_{t+1}(s_t) \,|\, \Filt_{t})$ holds. \hfill \break
 (c) By \eqref{eqju22}--\eqref{eqju23} and Lemma \ref{thmexpectedsoln},  the $\FF$-stopping times  $\tau^*_t$ corresponds to an optimal equilibrium of the single-period stochastic affine game $\AG(X_t,\EP(U_{t+1}\,|\,\Filt_t), G)$ and $U_t=\SOL \big(X_t,\EP(U_{t+1}\,|\,\Filt_t),G\big)$ is the corresponding value.

The statements (i) and (ii) will be established simultaneously by the backward induction. For $t=T$, we have that $\tau^*_T =(T,\ldots,T)$. The stochastic affine game $\ASG_T(X,G)$ is always stopped at time $T$ with the payoff vector $\bm{X}_T=\bm{U}_T=\bm{V}_T(\tau^*_T) = V^*_T$ also being the value. Let us now assume both statements are true for the game $\ASG_{t+1}(X,G)$, so its value is given by $\bm{V}_{t+1}^*=\bm{U}_{t+1}$.

We first prove part (i). Let us denote $\whts =\tau^{*1}_t\wedge\cdots \wedge\tau^{*m}_t$.   If $\whts =t$, the game is stopped at time $t$. By (c), $U_t$ is the payoff of $\tau^*_t$ in $\AG(X_t,\EP(U_{t+1}\,|\,\Filt_t), G)$. Since by the induction hypothesis $U_{t+1}=V^*_{t+1}$, from  property (a) it follows that $U_{t}=V_t(\tau^*_t)$.

If $\whts  \geq t+1$, then the game is not stopped at time $t$. By the definition of $\tau^*_t$, we have that $U^i_{t}>X^i_{t}$ for all $i\in\Mset$. From  property (c), we see that no one exercises in the optimal equilibrium of $\AG(X_t,\EP(U_{t+1}\,|\,\Filt_t), G)$ and thus $U_t=\EP( U_{t+1}\, |\, \Filt_t)$. Combining this with  property (b) and the induction hypothesis, we obtain
\[
U_t=\EP( U_{t+1}\, |\, \Filt_t)= \EP( V_{t+1}(\tau^*_t)\, |\, \Filt_t) = {V}_t(\tau^*_t).
\]

We now proceed to the proof of part (ii). From part (i), we know that
 \[
 \bm{V}_t(\tau^*_t)=\bm{U}_t=\SOL \big( X_t,\EP(U_{t+1}\,|\,\Filt_t),G\big)
 =\SOL \big( X_t,\EP(V_{t+1}^*\,|\,\Filt_t),G \big).
  \]
To check that $\tau^*_t$ is an optimal equilibrium, we require for each $k\in\Mset$,
\begin{gather}\label{eqju27}
V^k_{t}(\tau^{*k}_t , s^{-k}_t)\geq \SOL\big(X_t,\EP(V_{t+1}^*\,|\,\Filt_t),G\big)_k
\geq V^k_{t}(s^k_t, \tau^{*,-k}_t ), \quad \forall \, s^k_t \in\Strat^k_t ,\, \forall \, s^{-k}_t \in \Strat^{-k}_t .
\end{gather}
Let $s'=(\tau^{*k}_t , s^{-k}_t), s''=(s^k_t , \tau^{*,-k}_t )$ be alternative strategy profiles with the minimal stopping times $\widehat s'_t, \widehat s''_t$, respectively. We need to examine three cases.

\noindent{\it Case 1}:$\, $ If $\widehat s'_t=\widehat s''_t =t$, then, by property (a), both $s'_t$ and $s''_t$ can be interpreted as strategy profiles of the single-period  affine game $\AG(X_t,\EP({V}_{t+1}^*\,|\,\Filt_t), G)$. Hence the validity \eqref{eqju27} can be deduced from property (c), because $\SOL \big(X_t,\EP(V_{t+1}^*\,|\,\Filt_t),G \big)$ is the value of the game $\AG(X_t,\EP({V}_{t+1}^*\,|\,\Filt_t), G)$ for player $k$.

\noindent{\it Case 2}:$\, $ If $\widehat s'_t \geq t+1$, then, by property (b), we have that ${V}_t(s'_t)=\EP(V_{t+1}(s'_t)\,|\,\Filt_t)$.
By an application of the induction hypothesis, we obtain
\[
V^k_{t}(s'_t)=\EP \big(V^k_{t+1}(s'_t)\,\big|\,\Filt_{t}\big)\geq\E_\P \big(V^{*k}_{t+1}\,\big|\,\Filt_{t}\big).
\]
Since $\tau^{*k}_t \geq t+1$, property (c) tells us that player $k$ does not exercise in the optimal equilibrium of the game $\AG(X_t,\EP({V}_{t+1}^*\,|\,\Filt_t), G)$. Interpreting $\EP (V^{*k}_{t+1}\,|\,\Filt_{t})$ as the expected payoff of player $k$ if no one exercises, we have
\[
\SOL \big( X_t,\EP(V_{t+1}^*\,|\,\Filt_t),G \big)_k \leq \EP (V^{*k}_{t+1}\,|\,\Filt_{t}) \leq {V}^k_t(s'_t),
\]
as required.

\noindent{\it Case 3}:$\, $ If $\widehat s''_t \geq t+1$, then the argument is similar to Case 2. Again, by property (b) and the induction hypothesis, we obtain
\[
V^k_{t}(s''_t)=\EP \big(V^k_{t+1}(s''_t)\,\big|\,\Filt_{t}\big)\leq\E_\P \big(V^{*k}_{t+1}\,\big|\,\Filt_{t}\big).
\]
Since $\tau^{*i}_t \geq t+1$ for all $i\neq k$, property (c) tells us that none of the players from $\Mset\setminus\{k\}$ exercises in the optimal equilibrium of the game $\AG(X_t,\EP({V}_{t+1}^*\,|\,\Filt_t), G)$. Interpreting $\EP (V^{*k}_{t+1}\,|\,\Filt_{t})$ as the expected payoff of player $k$ if no one exercises, we get
\[
\SOL \big(X_t,\EP(V_{t+1}^*\,|\,\Filt_t),G\big)_k \geq \EP (V^{*k}_{t+1}\,|\,\Filt_{t}) \geq {V}^k_t(s'_t),
\]
as required. We conclude that in all three cases, \eqref{eqju27} is valid. Therefore, $U_t$ is the value and $\tau^*_t $ is an optimal equilibrium of $\ASG_t(X,G)$.
\hfill $\Box$

\subsection{Coalition Values}

 We have shown in Theorem \ref{thmjt01} that under certain conditions on $G$ in the affine game $\AG(X,P,G)$, the value $V^*_{\Aset}$ for a coalition of players $\Aset\subseteq\Mset$ is additive, meaning that $V^*_{\Aset}=\sum_{i\in\Aset} V^*_{i}$. It is not hard to generalize those arguments to the stochastic affine game $\ASG_t(X,G)$.

\begin{theorem} \label{thmju50}
Consider the multi-period stochastic affine game $\ASG_t(X,G)$ where $G$ is a $\Kmot$-matrix.
Suppose that the column sums of $G$ are non-negative, that is, condition \eqref{eqjt01} holds.
Then for any $\Aset\subseteq\Mset$, the value $V_t^{*\Aset}$ exists and
\[
V_t^{*\Aset}=\esssup_{s_t^\Aset\in\Strat^\Aset_t}\essinf_{s^{-\Aset}_t\in\Strat^{-\Aset}_t} V^\Aset_t(s^\Aset_t, s^{-\Aset}_t)=\essinf_{s^{-\Aset}_t\in\Strat^{-\Aset}_t} \esssup_{s^\Aset_t \in\Strat^\Aset_t}
V^\Aset_t(s^\Aset_t, s^{-\Aset}_t) = \sum_{i\in\Aset} V_t^{*i}.
\]
\end{theorem}

\begin{proof}
As in Theorem \ref{thmjt01}, we introduce the extended game with a dummy player with index 0, who does nothing and has the payoff function $V^{0}_t(s)=-\sum_{i\in\Mset} V^i_t(s_t)$ for all $s_t \in \Strat_t$. Suppose that $\tau^*_t$ is an optimal equilibrium of $\ASG_t(X,G)$. Our goal is to show that $\tau^*_t$ is also an optimal equilibrium of the extended game.
Since the dummy player cannot make any choices, it is clear that $\tau^*_t$ is a Nash equilibrium of the extended game and has the optimal equilibrium property for players $1,2,\ldots,m$.

To check that $\tau^*_t$ has also the optimal equilibrium property for the dummy player, it is enough to show that $V^{0}_t(\tau^*_t) = \essinf_{s_t \in\Strat_t} V^{0}_t(s_t)$ or, equivalently,
\begin{gather}\label{eqju52}
\sum_{i\in\Mset} V^{i}_t(\tau^*_t) = \sum_{i\in\Mset} V^{*i}_t = \esssup_{s_t \in\Strat_t} \sum_{i\in\Mset} V^{i}_t(s_t).
\end{gather}
We will establish \eqref{eqju52} using the backward induction. For $t=T$, property \eqref{eqju52} is trivially satisfied.
Let us now assume that it holds in the game $\ASG_{t+1}(X,G)$, so that
\begin{gather}\label{eqju53}
\EP \Big(\sum_{i\in\Mset} V^{*i}_{t+1} \,\Big|\, \Filt_t\Big) \geq
\EP \Big( \sum_{i\in\Mset} V^{i}_{t+1}(s_{t+1}) \,\Big|\,\Filt_t \Big)
 =\sum_{i\in\Mset} V^{i}_{t}(s_{t+1}),\quad \forall\,s_{t+1}\in\Strat_{t+1}.
\end{gather}
For any particular $s_t \in\Strat_t$, there are two cases to examine.

\noindent{\it Case 1}:$\, $ On the event $\{\widehat s_t = t\}$, we recall that the game on the time interval $[t,t+1]$ is equivalent to the single-period affine game $\AG(X_t, \EP(V^*_{t+1}\,|\,\Filt_t), G)$. In particular, its value is also given by $V_t(\tau^*_t)=V^*_t$. From the proof of Theorem \ref{thmjt01}, since the matrix $G$ satisfy the appropriate conditions, the optimal equilibrium property is preserved with the addition of the dummy player. Therefore, we must have
\begin{gather*}
\sum_{i\in\Mset} V^{i}_t(\tau^*_t) \geq \sum_{i\in\Mset} V^{i}_t(s_t ).
\end{gather*}
\noindent{\it Case 2}:$\, $ On the event $\{\widehat s_t \geq t+1\}$, a strategy profile $s_t$ can also be seen as a strategy in $\Strat_{t+1}$ and thus the induction hypothesis \eqref{eqju53} can be applied.
 Furthermore, since $\EP(V^*_{t+1}\,|\,\Filt_t)$ is the payoff vector of the single-period  affine game $\AG(X_t, \EP(V^*_{t+1}\,|\,\Filt_t), G)$ if no one exercises, using the same argument as before and the induction hypothesis \eqref{eqju53}, we obtain
\begin{gather*}
\sum_{i\in\Mset} V^{i}_t(\tau^*_t) \geq \sum_{i\in\Mset} \EP ( V^{*i}_{t+1} \,|\, \Filt_t) \geq \sum_{i\in\Mset} V^{i}_{t}(s_t).
\end{gather*}
In both cases, we have shown that the equality
\[
\sum_{i\in\Mset} V^{i}_t(\tau^*_t) = \sum_{i\in\Mset} V^{i}_{t}(s_t)
 \]
holds for all $s_t \in\Strat_t$. Hence \eqref{eqju52} is established and the induction is complete.
\end{proof}

\begin{remark}{\rm
From the proofs of Theorems \ref{thmmultiperiodvaluea} and \ref{thmju50}, it is clear that is possible to further generalize the game by making the matrix $G$ time-dependent and $\FF$-adapted,
so that we deal with a matrix-valued process $G_t,\, t=0,1, \dots , T-1$, where with probability one the random matrix
 $G_t (\omega )$ is  a $\Kmot$-matrix. In this version of the multi-period stochastic affine game, the payoff function would be
 given by the following expression
\begin{gather*}
V^k_{t}(s_t)=\EP\big( X^k_{\widehat s_t} \I_{\{k\in\Eset(s_t)\}}+\wh X^k_{\widehat s_t}\I_{\{k\notin\Eset(s_t)\}} \,\big|\, \Filt_t \big)
\end{gather*}
where
\begin{gather*}
\wh X^k_{\widehat s_t}=V^{*k}_{\widehat s_t+1}
+G^{k \Eset(s_t)}_{\widehat s_t} \big(G^{\Eset(s_t)\Eset(s_t)}_{\widehat s_t}\big)^{-1}
\big( X^{\Eset(s_t)}_{\widehat s_t} - V^{*\Eset(s_t)}_{\widehat s_t+1}\big),\quad \widehat s_t<T.
\end{gather*}
It is rather clear that if we adjust the definition of $U_t$  by setting $U_T=X_T$ and
\begin{gather*} 
\bm{U}_t=\SOL\big(X_t,\EP \big(\bm{U}_{t+1}\,\big|\,\Filt_{t}\big), G_t\big),\quad t =0,1,\dots ,T,
\end{gather*}
then the proofs (and thus also the conclusions) of Theorems \ref{thmmultiperiodvaluea} and \ref{thmju50} will still hold.}
\end{remark}

\section{Reflected BSDEs for Stochastic Affine Games} \label{sec5}

In this final section, we briefly examine the relationships between multi-period affine games and solutions to certain
multi-dimensional reflected BSDEs.

\subsection{A Class of Multi-Dimensional Reflected BSDEs}

Let us first recall the properties of variational inequalities, which we will need in what follows.
For an in-depth analysis of variational inequalities, we refer to Facchinei and Pang \cite{Facchinei}
and Harker and Pang \cite{Harker}.

\begin{definition} {\rm
Let $\D$ be a subset of $\R^m$ and $F\map{\D}{R^m}$ be a mapping. In the \emph{variational inequality problem} $\VI(\D,F)$,
the goal is to find a vector $z\in\D$ such that}
\begin{gather*}
(F(z))^{\mathsf T} (y-z) \geq 0, \quad \forall\, y\in\D.
\end{gather*}
\end{definition}

Typically, $\D$ is assumed to be closed and convex, while $F$ is assumed to be continuous.
If $F$ is an affine map $F(z)=q+Mz$, then $\VI(\R^m_+,F)$ is equivalent to $\LCP(q,M)$.
For $\LCP(q,M)$, we have seen in Proposition \ref{proplc03} that if $M$ is a $\Pmat$-matrix, then the existence and uniqueness of solutions is guaranteed. Similar existence and uniqueness results are also known for $\VI(\D,F)$. In particular, Proposition \ref{thmlp05} can be found in Harker and Pang \cite{Harker}.

\begin{proposition}\label{thmlp05}
(i) Let $\D$ be a non-empty, closed, convex subset of $\R^m$ and let $F\map{\D}{\R^m}$ be a continuous mapping.
If $F$ is \emph{strongly monotone} on $\D$, that is, there exists $c>0$ such that
\[
(F(x)-F(y))^{\mathsf T}(x-y) \geq c \, \norm{x-y}^2, \quad \forall\, x,y \in \D,\ x\neq y,
\]
then $\VI(\D,F)$ has a unique solution. \hfill \break
(ii) Let $O$ be a rectangular region in $\R^m$ and let $F\map{O}{\R^m}$ be a continuous mapping.
If $F$ is a \emph{uniform $\Pmat$-function} on $O$, that is, there exists $c>0$ such that
\[
\max_{1\leq i\leq m} (F_i(x)-F_i(y))(x_i-y_i) \geq c \, \norm{x-y}^2, \quad \forall\, x,y \in O ,\ x\neq y,
\]
then $\VI(O,F)$ has a unique solution.
\end{proposition}

Let us introduce a particular class of multi-dimensional reflected BSDEs with solutions in a rectangular region.
Assume that we are given $\FF$-adapted, $\R^m$-valued processes $L$ and $U$ satisfying $-\infty\leq L\leq U\leq \infty$ for $i=1,\ldots,m$. Then the random rectangular region bounded by $L$ and $U$ is given by, for $t=0,\ldots,T$,
\[
\Oo_t=\big\{(Y^1_t ,\ldots,Y^m_t )\in\R^m : L^i_t\leq Y^i_t \leq U^i_t,\  1\leq i \leq m\big\}
\]
where $Y^1_t , \dots , Y^m_t $ are $\Filt_t$-measurable random variables.

We consider the following $m$-dimensional reflected BSDE with the data $(\xi, F, N, L, U, G)$
\begin{gather}\label{eqld02}
Z_t = \xi + \sum_{u=t}^{T-1} f(u,Z_u,\phi_u) + J_T-J_t - \sum_{u=t}^{T-1} \phi_u \Delta N_{u+1}
\end{gather}
where
\\ (i) $N$ is an $\R^d$-valued, $(\P,\FF)$-martingale with the predictable representation property;
\\ (ii) $f$ is an $\FF$-adapted, $\R^m$-valued random map;
\\ (iii)  $G$ is an $\FF$-adapted, $\R^{m}$-valued random map;
\\ (iv) $L$ and $U$ are $\FF$-adapted $\R^m$-valued processes with $L^i\leq U^i$ for $i=1,\ldots,m$;
\\ (v) $\xi$ is an $\Filt_T$-measurable random variable in $\R^m$ with $L_T^i\leq \xi^i \leq U_T^i$ for $i=1,\ldots,m$.

\begin{definition}\label{defld01} {\rm A \emph{solution of the reflected BSDE} \eqref{eqld02}
 with data $(\xi, F, N, L, U, G)$ is a triplet $(Z,\phi,J)$ of processes such that: \hfill \break (a) $Z$ is an $\FF$-adapted, $\R^m$-valued process such that equation \eqref{eqld02} is satisfied for every $t =0,\dots, T$; in particular, $Z_T=\xi$,
\\ (b) the condition $Z_t \in \Oo_t$ holds for every $t=0, \dots ,T$;
\\ (c) $\phi$ is an $\FF$-adapted,  $\mathbb{R}^{m \times d}$-valued process;
\\ (d) $J$ is an $\FF$-predictable, $\mathbb{R}^{m}$-valued process satisfying
\begin{gather}\label{eqld03}
J_t=\sum_{u=0}^{t-1} G(u,Z_u,\phi_u,\Delta K_{u+1})
\end{gather}
where the $\FF$-predictable, $\R^m$-valued process $K$ satisfies, for $t=0,1,\ldots,T-1$,
\begin{gather}\label{eqld04}
(\Delta K_{t+1})^{\mathsf T} (Z'-Z_t) \geq 0 \quad \forall\,Z'\in\Oo_t ,
\end{gather}
where $\Delta K_{t+1} := K_{t+1}- K_{t}$.}
\end{definition}

It is worth noting that one can replace \eqref{eqld04} by
\[
\sum_{t=0}^{T-1} (\I_{\{Z_t>L_t\}})^{\mathsf T} \Delta K^+_{t+1}
=\sum_{t=0}^{T-1} (\I_{\{Z_t<U_t\}})^{\mathsf T} \Delta K^-_{t+1}=0
\]
where $\Delta K^{i+}_{t+1}=\Delta K^i_{t+1} \I_{\{\Delta K^i_{t+1}>0\}}$ and $\Delta K^{i-}_{t+1}=-\Delta K^i_{t+1} \I_{\{\Delta K^i_{t+1}<0\}}$ for each $i=1,\ldots,m$.


Let us make some comments on the existence and uniqueness of a solution to BSDE \eqref{eqld02}.
Assume that the function $G(t,z,y,k)$ has an inverse $G^{-1}_{t,z,y}$ with respect to $k$. Then, using the standard backward induction arguments, we may reduce the reflected BSDE from Definition \ref{defld01} to the variational inequality problem $\VI(\Oo_t,F)$, where the mapping $F$ is defined by the following expression
\[
F(z):=G^{-1}_{t,z,\phi_t}(z-f(t,z,y)-p).
\]
It is worth noting that here $\phi_t$ has already been computed from the induction step. In view of part (i) in Proposition \ref{thmlp05}, it suffices to postulate that $F$ is strongly monotone to guarantee the existence and uniqueness of a solution. However, since the domain $\Oo_t$ is now a rectangular region, rather than an arbitrary random convex subset of $\R^m$, we may also use the weaker condition from part (ii) in Proposition \ref{thmlp05}, that is, to assume that $F$ is a uniform $\Pmat$-function.

\subsection{Game Value as a Solution to the Reflected BSDE}

The connection between two-player Dynkin games and doubly-reflected BSDEs is well known (see, for instance, the seminal paper by Cvitani\'c and Karatzas \cite{Cvitanic}). We will now apply the multi-dimensional reflected BSDE from the previous section to a multi-player stochastic affine game. In fact, the full strength of this approach will only become clear
when dealing with affine game options in a financial market model with frictions. In the present simple framework, the underlying martingale $N$ introduced in Definition \ref{defld01} plays no essential role and thus it will be enough to consider the reduced-form BSDE, as given by \eqref{eqlm10}. In the financial applications, the martingale $N$ is given in advance, since it represents the discounted prices of traded assets, and the $\mathbb{R}^{m \times d}$-valued process $\phi $ is interpreted as a collection of hedging strategies
for $m$ parties.

As shown in Section \ref{secLCP}, the single-period affine game can be solved by focussing on a particular linear complementarity problem. Since linear complementarity problems are special cases of variational inequalities, it is not surprising that the value process of the stochastic affine game can be computed by solving a reflected BSDE stemming from the variational inequality.
Let us thus consider the $m$-player  multi-period stochastic affine game $\ASG (X,G)$  introduced in Definition \ref{defmultiperiodgame}. In particular, we restrict ourselves to the case of a (non-singular) $\Kmat$-matrix $G$, since linear complementarity problems with singular matrices may fail to possess a solution. Suppose that $X$ is an  $\FF$-adapted, $\R^m$-valued process and $G\in\R^{m\times m}$ is a $\Kmat$-matrix.
We now consider the reduced-form reflected $\BSDE (X,G)$, which is formally obtained by setting $f=0$ and $\xi = X_T$ in \eqref{eqld02} and by taking the conditional expectation with respect to $\Filt_t$
\begin{gather}\label{eqlm10}
Z_T=X_T , \quad Z_t - J_{t+1}+J_t = \EP(Z_{t+1}\, | \, \Filt_t).
\end{gather}
We now specialize Definition \ref{defld01} to the present situation where the mapping $J$ is given
by \eqref{eqlm11} and the random orthant $\Oo_t$ is defined by,  for $t=0,\ldots,T$,
 \[
\Oo_t=\big\{(Y^1_t,\ldots,Y^m_t ) \in \R^m : Y^i_t \geq X^i_t,\  1\leq i \leq m\big\}
\]
where $Y^1_t, \dots , Y^m_t$ are $\Filt_t$-measurable random variables.

\begin{definition}\label{deflm30}  {\rm
A \emph{solution to the reflected BSDE} \eqref{eqlm10} is a pair $(Z,J)$ which satisfies: \hfill \break
(i) $Z$ is an $\FF$-adapted, $\mathbb{R}^{m}$-valued process satisfying $Z_t\in \Oo_t$, \hfill \break
(ii) $J$ is an $\FF$-predictable, $\mathbb{R}^{m}$-valued process satisfying
\begin{gather} \label{eqlm11}
J_t=\sum_{u=0}^{t-1} G \Delta K_{u+1}
\end{gather}
where $K$ is an $\FF$-predictable, $\R^m$-valued, non-decreasing process satisfying}
\begin{gather*}
\sum_{t=0}^{T-1} ( \I_{\{Z_t>X_t\}})^{\mathsf T} \Delta K_{t+1}=0.
\end{gather*}
\end{definition}

The next lemma shows that the reflected $\BSDE (X,G)$ has a unique solution.

\begin{lemma} \label{lemmacc}
The reflected $\BSDE (X,G)$ given by formula \eqref{eqlm10} has a unique solution pair $(Z,J)$.
Moreover, the solution $Z$ may be written as $Z_T=X_T$ and, for all $t=0,1, \dots ,T-1$,
\begin{gather}\label{eqci23}
 Z_t=\SOL \big(X_t, \EP(Z_{t+1}\,|\, \Filt_t), G \big).
\end{gather}
\end{lemma}

\begin{proof}
We apply standard backward induction arguments and we note that,
for each $t$, we deal with an $\Filt_t$-measurable problem, which can be solved for each $\omega $.
It is thus clear that solving the $\BSDE(X,G)$ over $[t,t+1]$ reduces to the following deterministic problem
\begin{gather*}
z-Gk=p,\quad z\geq x,\quad k\geq0, \quad \I_{\{z>x\}} k=0,
\end{gather*}
where $z=Z_t,\, k=\Delta K_{t+1},\, x=X_t$. We observe that this problem is equivalent to $\LCP(p-x,G)$ (or $\VI(\Oo_t,F)$ where $F(z)=G^{-1}(z-p)$). Since $G$ is a $\Kmat$-matrix (and thus also $\Pmat$-matrix), by Proposition \ref{proplc03}, there exists a unique solution $z$.
\end{proof}

For the last result, it suffices to combine Lemma \ref{lemmacc} with Theorem \ref{thmmultiperiodvaluea}.

\begin{corollary}
Let $V^*$ be the value process of $\ASG(X,G)$ with a $\Kmat$-matrix $G$.
Then $V^*= Z $ where the pair $(Z,J)$ is the unique solution to the reflected $\BSDE (X,G)$.
\end{corollary}

\newpage


\noindent {\bf Acknowledgement.}
The research of Ivan Guo and Marek Rutkowski was supported under Australian Research Council's
Discovery Projects funding scheme (DP120100895).


\vskip 5 pt
\section{Appendix: Redistribution Games as Affine Games} \label{secredd}

 The goal of the appendix is to re-examine a \emph{general redistribution game} $\GRG(X,P,\alpha )$, where $\alpha = (\alpha_1, \dots , \alpha_m)$ with $\alpha_i> 0$ and $\sum_{i=1}^m \alpha_i < 1$, which was introduced in Section 3 of \cite{Guo2}.
Our goal is to show that such a game is a special case of an affine game given by Definition
\ref{defjj08} with a non-singular matrix $G$.

 Let the matrix $\whD$ be given by the following expression
\begin{gather}
\whD := \begin{pmatrix} \label{ceqjp02}
\alpha_1-\alpha_1^2 & -\alpha_1\alpha_2 & \cdots & -\alpha_1\alpha_m\\
-\alpha_2\alpha_1 & \alpha_2-\alpha_2^2 & \cdots & -\alpha_2\alpha_m\\
\vdots & \vdots & \ddots & \vdots\\
-\alpha_m\alpha_1 & -\alpha_m\alpha_2 & \cdots & \alpha_m-\alpha_m^2
\end{pmatrix}.
\end{gather}

\newpage

\begin{lemma} \label{appgrg}
The general redistribution game $\GRG(X,P,\alpha )$ is identical to the affine game $\AG(X,P,\whD)$.
\end{lemma}

\begin{proof}
One can check that the inverse $D$ of $\whD$ is given by
\begin{gather}\label{eqjp03}
D = \frac{1}{1-\sum_{i=1}^m \alpha_i} \begin{pmatrix}
1+\frac{1-\sum_{i=1}^m \alpha_i}{\alpha_1} & 1 & \cdots & 1\\
1 & 1+\frac{1-\sum_{i=1}^m \alpha_i}{\alpha_2} & \cdots & 1\\
\vdots & \vdots & \ddots & \vdots\\
1 & 1 & \cdots & 1+\frac{1-\sum_{i=1}^m \alpha_i}{\alpha_m}
\end{pmatrix}.
\end{gather}
Recall the payoff in $\AG(X,P,\whD)$ is given in Definition \ref{defjj08} as
\[
V(s)= P+\whD_{\cdot \Eset(s)} \big(\whD_{\Eset(s)\Eset(s)}\big)^{-1}(X_{\Eset(s)}-P_{\Eset(s)}).
\]
The inverse $\big(\whD_{\Eset(s)\Eset(s)}\big)^{-1}$ is similar to $D$, except it only contains $\alpha_i$ where $i\in\Eset(s)$ and $1-\sum_{i=1}^m \alpha_i$ is replaced by $1-\sum_{i\in\Eset(s)} \alpha_i$. It is clear that for $i\in\Eset(s)$, we have $V_i(s)=X_i$.

For $k\notin\Eset(s)$, we obtain
\begin{align*}
V_k(s)&=P_k+ \whD_{k\Eset(s)} \big(\whD_{\Eset(s)\Eset(s)}\big)^{-1}(X_{\Eset(s)}-P_{\Eset(s)})\\
&=P_k+\frac{\sum_{i\in\Eset(s)}(-\alpha_k\alpha_i) \Big(\frac{1-\sum_{j\in\Eset(s)} \alpha_j}{\alpha_i}(X_i-P_i) + \sum_{j\in\Eset(s)} (X_j-P_j)\Big)}{1-\sum_{i\in\Eset(s)} \alpha_i}\\
&=P_k-\frac{\alpha_k}{1-\sum_{i\in\Eset(s)} \alpha_i}  \Big(\Big(1-\sum_{j\in\Eset(s)} \alpha_j\Big)\sum_{i\in\Eset(s)}(X_i-P_i) + \sum_{i\in\Eset(s)}\alpha_i\sum_{j\in\Eset(s)} (X_j-P_j)\Big)\\
&=P_k-\frac{\alpha_k}{1-\sum_{i\in\Eset(s)} \alpha_i} \sum_{i\in\Eset(s)}(X_i-P_i).
\end{align*}
Recall from \cite{Guo2} (see Assumption (A.1) therein) that the {\it weights} $w_k(\Eset(s))$ are given by the equality $w_k(\Eset(s))=\frac{\alpha_k}{1-\sum_{i\in\Eset(s)} \alpha_i}$. Therefore, the payoff $V_k(s)$ can be represented as follows
\begin{align*}
V_k(s)=
\begin{cases}
 X_{k}, \ \  & k\in\Eset(s),\\
 P_k - w_k(\Eset(s)) \sum_{i\in\Eset(s)}(X_i-P_i), \ \ & k\in \Mset\setminus\Eset(s),
\end{cases}
\end{align*}
so that it matches the payoff function of $\GRG(X,P,\alpha)$ (see Definition 5 in \cite{Guo2}).
\end{proof}

We already know from Theorem 3.15 in \cite{Guo2} that a general redistribution game $\GRG(X,P,\alpha)$ is a WUC game with a unique value. To reaffirm these properties using Theorem \ref{thmjj30}(iii), it suffices to show that the matrix $\whD$ of a
 general redistribution game is $\Kmat$-matrix.

\begin{proposition}
The matrix $\whD$ given by \eqref{ceqjp02} is a $\Kmat$-matrix. The value $V^*$ of $\GRG(X,P,\alpha)=\AG(X,P,\whD)$
satisfies $V^*=\pi^{D}_{\Oo (X)}(P)$ where the projection $\pi^{D}_{\Oo (X)} : \R^m \to \Oo (X) $ is taken
under the inner product $\langle x, y\rangle^{D} = x^{\mathsf T} D y$ given by
\begin{align} \label{innpro}
x^{\mathsf T} D y=\sum_{i=1}^m \left( \frac{x_iy_i}{\alpha_i}\right) + \frac{\left(\sum_{i=1}^m x_i \right)\left(\sum_{i=1}^m y_i\right)}{1-\sum_{i=1}^m \alpha_i}.
\end{align}
\end{proposition}

\begin{proof}
It is clear that $\whD$ is a $\Zmat$-matrix because $\alpha_i(1-\alpha_i)>0$ and $-\alpha_i\alpha_j<0$.
To show that $\whD$ is also a $\Pmat$-matrix, it suffices to check that $\det \whD >0$, since its principal sub-matrices have the same structure. Note that we can write $\det \whD = \det (I-B)\prod_{i=1}^m \alpha_i$ where $B$ is given by
\[
B := \begin{pmatrix}
\alpha_1 & \alpha_1 & \cdots & \alpha_1\\
\alpha_2 & \alpha_2 & \cdots & \alpha_2\\
\vdots & \vdots & \ddots & \vdots\\
\alpha_m & \alpha_m & \cdots & \alpha_m
\end{pmatrix},
\]
 so that $B$ is a rank one matrix with a non-zero eigenvalue given by $\textrm{tr}(B)= \sum_{i=1}^m \alpha_i$. Hence the characteristic polynomial of $B$ is
\[
\det (\lambda I-B)=\lambda^{m-1}\Big(\lambda - \sum_{i=1}^m \alpha_i \Big).
\]
By setting $\lambda=1$, we obtain
\[
\det \whD = \det (I-B)\prod_{i=1}^m \alpha_i= \Big(1-\sum_{i=1}^m \alpha_i\Big)\prod_{i=1}^m \alpha_i > 0.
\]
Therefore, $\whD$ is indeed a $\Kmat$-matrix. Since $\whD$ is symmetric, it is also positive definite.
From Theorem \ref{thmjj30}(ii), we deduce that the value of $\GRG(X,P,\alpha)=\AG(X,P,\whD)$ can also be written as projection under the inner product $\langle x, y\rangle^{D} = x^{\mathsf T} D y$ where $D =  \wh{D}^{-1}$. It is easy to check that
the inner product $\langle x, y\rangle^{D}$ is given by \eqref{innpro}.
As expected, this expression coincides with the inner product introduced in \cite{Guo2} (see formula (3.11) therein)
thus confirming also the projection results obtained in \cite{Guo2}.
\end{proof}

\end{document}